\documentclass[a4paper,11pt,twoside]{article}

\usepackage[latin1]{inputenc} 
\usepackage[T1]{fontenc} 
\newlength{\minipagewidth}
\usepackage{a4wide}
\usepackage{amsmath,amsthm}
\usepackage{amssymb}
\usepackage{graphicx}
\usepackage{color}
\usepackage{epic}
\usepackage{color}
\usepackage{enumerate}
\usepackage{empheq}

\newcommand{\pN}[1]{{#1}}
\newcommand{\opt}[1]{\widetilde{#1}}
\newcommand{\cc}{{\cal C}_2}
\newcommand{\al}{\mathcal{A}}

\newcommand{\E}{\mathbb{E}}

\newcommand{\R}{\mathbb{R}}

\renewcommand{\S}{\mathbb{S}}

\renewcommand{\P}{\mathbb{P}}
\renewcommand{\L}{\mathbb{L}}

\newcommand{\calL}{\mathcal{L}}
\newcommand{\calP}{\mathcal{P}}

\newcommand{\calN}{\mathcal{N}}
\newcommand{\calR}{\mathcal{R}}
\newcommand{\calW}{\mathcal{W}}

\newcommand{\eps}{\varepsilon}
\newcommand{\ph}{\varphi}

\newcommand{\Span}{\mathrm{Span} }
\newcommand{\Tr}{\mathrm{Tr}}
\newcommand{\Id}{\mathrm{Id}}
\newcommand{\Law}{{\rm Law}}

\newcommand{\limop}[1]{\mathop{ {\rm #1} }\limits}

\newcommand{\e}{{\rm e}}
\renewcommand{\d}{ {\rm d}}
\newcommand{\as}{\quad {\rm a.s.}}

\newcommand{\eqdef}{ \mathop{=}^{{\rm def}} }
\newcommand{\one}[1]{ {\rm l} \hspace{-.7 mm} {\rm l}_{ #1}   }
\newcommand{\dps}{\displaystyle}
\newcommand{\fracd}[2]{\frac{\dps #1}{\dps #2}}
\newcommand{\cond}[1]{  { \, \left | \, #1 \right .} }
\newcommand{\st}{ \, | \, }

\newcommand{\abs}[1]{\left | #1\right |}
\newcommand{\set}[1]{\left\{#1\right\}}
\newcommand{\pare}[1]{ \left(#1\right) }
\newcommand{\p}[1]{ \left(#1\right) }

\newcommand{\bracket}[1]{\left \langle #1\right \rangle}

\newcommand{\barr}[1]{\left. \begin{array}{#1}}
\newcommand{\earr}{\end{array}\right.}
\newcommand{\bmat}{\begin{pmatrix}}
\newcommand{\emat}{\end{pmatrix}}
 
\theoremstyle{plain}
\newtheorem{The}{Theorem}[section]
\newtheorem{Lem}[The]{Lemma}
\newtheorem{Pro}[The]{Proposition}

\newtheorem{Def}[The]{Definition}

\numberwithin{equation}{section}

\theoremstyle{definition}
\newtheorem{Rem}[The]{Remark}

\title{A $N$-uniform quantitative Tanaka's theorem for the conservative Kac's $N$-particle system with Maxwell molecules}

\author{Mathias Rousset\footnote{Supported in part by {\it ERC MSMath}. {\it AMS 2000 subject classifications:} Primary 60J27; secondary 65C40. {\it Keywords}: trend to equilibrium, Markov process, Kac's particle system, coupling. } \\
{\it CERMICS } \\
{\it INRIA Paris - Rocquencourt, Ecole des Ponts Paristech, \& Université Paris-Est. }
}

\begin{document}

\maketitle

\begin{abstract}
This paper considers the space homogenous Boltzmann equation with Maxwell molecules and arbitrary angular distribution. Following Kac's program, emphasis is laid on the the associated conservative Kac's stochastic $N$-particle system, a Markov process with binary collisions conserving energy and total momentum. An explicit Markov coupling (a probabilistic, Markovian coupling of two copies of the process) is constructed, using simultaneous collisions, and parallel coupling of each binary random collision on the sphere of collisional directions. The euclidean distance between the two coupled systems is almost surely decreasing with respect to time, and the associated quadratic coupling creation (the time variation of the averaged squared coupling distance) is computed explicitly. Then, a family (indexed by $\delta > 0$) of $N$-uniform ``weak'' coupling / coupling creation inequalities are proven, that leads to a $N$-uniform power law trend to equilibrium of order ${\sim}_{ t \to + \infty} t^{-\delta} $, with constants depending on moments of the velocity distributions strictly greater than $2(1 + \delta)$. The case of order $4$ moment is treated explicitly, achieving Kac's program without any chaos propagation analysis. Finally, two counter-examples are suggested indicating that the method: (i) requires the dependance on $>2$-moments, and (ii) cannot provide contractivity in quadratic Wasserstein distance in any case.
\end{abstract}

\paragraph{Foreword} This paper is the rewritten, submitted version of the preliminary version: M. Rousset, {\it Scalable and Quasi-Contractive Markov Coupling of Maxwell Collision} also available on arXiv. The latter preliminary version is \emph{not} to be published.

\tableofcontents

\section*{Introduction}
\paragraph{Kac's particle system}~

The Kac's conservative $N$-particle system is a Markov stochastic process describing the evolution of the velocities in $(\R^d)^N$ of $N$ particles subject to random binary elastic collisions. The latter satisfies: (i) independence with respect to possible positions of particles (space homogeneity); (ii) conservation of momentum and kinetic energy (elasticity). Assuming propagation of chaos (asymptotic independence of subsets of particles), the formal $N \to \infty$ limit of the probability distribution of a single velocity satisfies the classical Boltzmann kinetic non-linear equation, in its space homogenous simplified form. When the collision rate is constant, in particular independent of the relative speed of particle pairs, we speak of \emph{Maxwell molecules}. 

A classical problem consists in quantifying the speed at which the probability distribution of the latter process converges towards its large time limit. The latter limiting, invariant probability, is the uniform distribution on the sphere $\S^{Nd-d-1}$ defined by the conservation of momentum and energy. 

\paragraph{Context}~

Let us recall standard strategies for probability flows solutions of evolution (linear or non-linear) equations defined by a reversible Markovian mechanism.

\begin{enumerate}[(i)]
\item  The \emph{(relative) entropy dissipation} method. Denote $t \mapsto \pi_t$ a probability flow with state space $E$, expected to converge to $\pi_\infty$. The entropy method computes the variation of the relative entropy
\[
\frac{\d }{\d t} \underbrace{ \int_E \frac{\d \pi_t }{\d \pi_\infty} \ln \frac{\d \pi_t }{\d \pi_\infty}  \d \pi_\infty }_{E(\pi_t)}= - D(\pi_t) \leq 0,
\]
and try to obtain exponential convergence to equilibrium by obtaining a so-called \emph{modified log-Sobolev inequality} of the form
\begin{equation}
  \label{eq:mLSI}
  E(\pi) \leq \frac{1}{2c_{\rm ls}} D(\pi) \qquad \forall \pi \in  \calP(E),
\end{equation}
for some constant $c_{\rm ls} >0$. When $t \mapsto \pi_t$ is the distribution flow of a reversible diffusion on a Riemannian manifold, the famous curvature condition $CD(c_{\rm cd},\infty)$ of Bakry and Emery (a mixture of strong convexity of the diffusion drift's potential, and uniform positive curvature of the metric, with lower bound $c_{\rm cd} > 0$, see~\cite{BakGenLed14} and references therein) yields such an exponential convergence by proving the inequality $0 < c_{\rm cd} \leq c_{\rm ls} $. The latter is obtained using the inequality
\[
\frac{\d}{\d t} D(\pi_t) \leq - 2 c_{\rm cd} D(\pi_t),
\]
and integrating through time. This topic has received considerable interest recently, due to the following gradient's flow interpretation: the probability flow of a reversible diffusion on a manifold is in fact a gradient flow of the relative entropy $E(\pi)$ with respect to the probability metric given by the quadratic Wasserstein distance $W_2$ (see the monographs~\cite{Vil08,AmbGigSav06}). This has led to the interpretation of $c_{\rm cd}$ has a uniform displacement convexity constant, and yielded a conceptual explanation for the inequality $0 < c_{\rm cd} \leq c_{\rm ls} $.

\item The weaker ($0 < c_{\rm ls} \leq c_{\rm sg} $) \emph{spectral gap} method which computes
\[
\frac{\d }{\d t} \underbrace{ \int_E \p{ 1-\frac{\d \pi_t }{\d \pi_\infty} }^2\d \pi_\infty }_{E_2(\pi_t)} = - D_2(\pi_t) \leq 0 ,
\]
and try to obtain exponential convergence to equilibrium by obtaining a so-called \emph{spectral gap inequality} of the form
\[
E_2(\pi) \leq \frac{1}{c_{\rm sg}} D_2(\pi) \qquad \forall \pi \in  \calP(E).
\]
When $t \mapsto \pi_t$ is the flow of a reversible Markov process, the latter is indeed the spectral gap of $D_2$ (the so-called Dirichlet form) seen as a self-adjoint operator in $\L_2(E,\pi_\infty)$.
\item The \emph{Markov coupling} method, which amount to construct an explicit Markov coupling, a probabilistic coupling of two copies of the Markov process of interest which is itself again Markov:
\[
t \mapsto (U_t,V_t) \in E \times E.
\]
If the latter coupling contracts with respect to some distance in an average $L_p$ sense:
\[
\frac{\d}{\d t} \E \p{ d(U_t,V_t)^p }^{1/p} \leq - c_{p} \, \E \p{ d(U_t,V_t)^p }^{1/p}
\] 
for any initial condition, then the method yields an upper bound on the contractivity (with constant $ c_{w_p} \geq c_{ p } > 0$)  with respect to the related probability Wasserstein distance $W_p$.  Exponential trend to equilibrium follows. Here again, for reversible diffusion on a manifold, the curvature condition $CD(c_{\rm cd},\infty)$ is typically required (see \cite{Wan97} and references therein) to construct such a { \bf contractive coupling using parallel transport}. Using again the gradient's flow interpretation, the $CD(c_{\rm cd},\infty)$ condition is essentially equivalent to contractivity in quadratic Wasserstein distance $c_{\rm cd} = c_{w_2}$, (see for instance~\cite{RenStu05,AmbGigSav06}).
\end{enumerate}
\begin{Rem}~
  \begin{enumerate}[(i)]
  \item Up to our knowledge, there is no such settled general theory for jump processes, yielding inequalities analogous to $c_{\rm cd}=c_{w_2} \leq c_{\rm ls}$, (see however~\cite{Oll09,Maa11} and related papers for recent approaches). This is of importance in our context, since space homogenous Boltzmann's collisions are reversible {\bf jump } processes, and an the type of angular distribution (``small jumps'') is known to influence the trend to equilibrium (see discussion below).
\item In the context of $N$-particle systems, such methods may be used either on the particle system (and one may look for $N$-uniform constants), or directly on the non-linear mean-field (here, kinetic) equation, seen as the formal $N=+\infty$ case.
\item In practice, some more or less weakened versions of the above inequalities, especially of the modified log-Sobolev (``entropy / entropy dissipation'') inequality are obtained. They are of the form:
\begin{equation}
  \label{eq:mLSI_weak}
\boxed{
  E(\pi)^{1+1/\delta} \leq \frac{1}{2c_{{\rm ls},\delta}(\pi) } D(\pi)\qquad \forall \pi \in  \calP(E),
}
\end{equation}
and yields { \bf algebraic or power law trends} of order $t^{-\delta}$, for $\delta \in ]0,+\infty]$ ($\delta=+\infty$ formally stands for the exponential case). $c_{{\rm ls},\delta}(\pi)$ is typically {\bf dependent on  moments } of $\pi$. A priori moment propagation estimates on the probability flow have to be obtained separately in order to get quantitative convergence to equilibrium. In particular, in kinetic theory, ``Cercignani's conjecture'' refers to the case $c_{{\rm ls},\delta=+\infty}(\pi) > 0$, where the type of dependence with respect to $\pi$ (moments, regularity) is known to be propagated by the probability flow. Usually, probabilists speak of modified log-Sobolev inequalities when $c_{{\rm ls}}$ is unconditionally bounded below (independent of $\pi$).
\end{enumerate}
\end{Rem}

The mathematical literature studying the convergence to equilibrium of the space homogenous Boltzmann kinetic equation, and its related Kac's conservative $N$-particle system is extremely vast, and we refer to the classical reviews~\cite{Cer69,Vil02}. In the present work, the so-called Boltzmann collision kernel (the Markov generator defining random collisions on the sphere defined by the conservation laws of a pair velocities) will be denoted with Carleman's representation
\begin{equation}\label{eq:kernel}
\boxed{
  b( v-v_\ast , \d n'_v ) \equiv {\rm unif}_{\theta}(n_v, \d n'_v) \beta(\abs{v-v_\ast} , \d \theta ),
}
\end{equation}
where in the above $(v,v_\ast)\in \R^d \times \R^d$ are the incoming (pre-collisional) velocities of a pair particle, $(n_v,n'_v) = \pare{\frac{v-v_\ast}{\abs{v-v_\ast}},\frac{v'-v'_\ast}{\abs{v'-v'_\ast}}} \in \S^{d-1} \times \S^{d-1} $ are respectively the pre- and post-collisional directions, $n_v \cdot n'_v = \cos \theta$ defines the scattering angle $\theta \in [0, \pi]$, and ${\rm unif}_{\theta}(n_v, \d n'_v)$ is the uniform probability distribution on the sphere defined by a prescribed scattering angle. In~\cite{DesMouVill11}, a review is provided about the different types of collisions classifying the possible large time behaviors. Following their convention, one can introduce two parameters $(\gamma,\nu) \in [-d, +\infty [ \times ]- \infty ,2]$ and assume that
\[
\beta(\abs{v-v_\ast} , \d \theta ) \limop{\sim} \abs{v-v_\ast}^{\gamma} \theta^{-\nu-1} \d \theta,
\] 
in the limit where $ \abs{v-v_\ast} \to +\infty $, as well as $\theta \to 0^+$. From physical scattering theory, the case $ \gamma < 0$ is often called ''soft potential'', while $\gamma > 0$ is called ''hard potential'' and $\gamma = 0 $ is called ``Maxwell molecules''. The case $\nu \leq 0$ corresponds usually to bounded kernels and is called ``angular cut-off''. The case $\nu \in ]0,2[$ corresponds to Levy generators associated to fractional pseudo-differential operators, while formally the case $\nu =2$ is the diffusive case, the Boltzmann operator becoming proportional to the Laplace-Beltrami operator on the sphere of collisional directions, also called the ``Landau operator''. Finally, we also mention the important Kac's ``caricature'' case, where $d=1$ and momentum is not conserved; the latter is usually considered with Maxwell molecules and angular cut-off: $\nu <0, \gamma =0$.

\paragraph{Literature}~

\emph{Entropy method, $N=+\infty$.} First, the most studied method for trend to equilibrium in kinetic theory is by far the entropy method, in the case of the kinetic ($N=+\infty$) equation. Some famous counterexamples (see \cite{Bob88,BobCer99,Vil03}) have shown that a weak entropy-entropy dissipation inequality of the form~\eqref{eq:mLSI_weak} (called ``Cercignani's conjecture'' in kinetic theory) cannot not hold for $\delta=+\infty$ ({\it i.e. } $c_{{\rm ls},\delta=+\infty}(\pi)=0$), even when restricting to reasonable conditions on $\pi$ (moments, regularity,...). This counter-example contains several physically realistic collisions (for instance $\nu < 0, \gamma < 2$, which includes Maxwell molecules in the non-diffusive case, of interest here). In fact, it has been conjectured in~\cite{DesMouVill11} from rigorous proofs in meaningful particular examples (for instance the Landau case  $\gamma = 0, \nu = 2$ in~\cite{DesVil00}), that a necessary and sufficient criteria such that a modified log-Sobolev inequality holds is the following: $c_{ {\rm ls} } > 0 \Leftrightarrow \gamma + \nu^+ \geq 2 $. The latter suggests a common contribution of the probability of high energy collisions ($\gamma$ large), and of small scattering angle ($\nu$ large). 

Meanwhile, many studies have been developped in the cases where exponential entropy convergence is known to fail, say $c_{\rm ls} = 0$. Some weakened versions of the ``entropy / entropy dissipation'' analysis of the form~\eqref{eq:mLSI_weak} (here is a sample: \cite{Cer82,CarCar92,BobCer99,TosVil99_trend,Vil03}) in order to obtain algebraic or power law trends with some a priori estimates on $\pi$ that has to be obtained seperately. 

These facts {\bf motivates the power-law behavior and moment dependence in the Maxwell case with angular cut-off $(\gamma = 0, \nu < 0)$, which are obtained in the present paper}.

\emph{Spectral gap and Wild's expansion, $N=+\infty$.}
For the case of interest in the present paper (Maxwell molecules, jump kernels: $(\gamma = 0, \nu < 0)$), an expansion method, known as Wild's expansion, enables to give precise estimates using some refined form of the central limit theorem. It has been shown in~\cite{CarLu03}, that arbitrary high moments of a velocity distribution necessarily lead to arbitrary slow decay to equilibrium (in $L^1$). In~\cite{CarLu03,DolGabReg09,DolReg10,DolReg12} a full theory of convergence to equilibrium for Maxwell molecules is then developed using Wild's method, showing that the convergence is essentially exponential with rate given by the spectral gap, but requires some moment and regularity condition on the initial condition, and a constant which is sub-optimal for short time.

In~\cite{Mou06}, the case of hard potentials is treated with a spectral method that essentially prove exponential convergence with rate given by the spectral gap of the linearized near equilibrium equation, and rely on moment creation in the case of hard potentials.

These facts motivates moment dependence found in the present paper, but suggests that the associated power law behavior is sub-optimal for very large times.

\emph{Spectral gap method $N<+\infty$.} Direct studies of the trend of equilibrium of the Kac's $N$-particle system have been undertaken~\cite{DiaSal00,CarCarLos03,CarCarLos08,Oli09}. The main striking feature of the latter list is the { \bf difficulty to achieve the so-called ``Kac's program'' for large time behavior} (see~\cite{MisMou11} ): obtaining a scalable ($N$-uniform) analysis of the trend to equilibrium of the $N$-particle system. A famous result (see ~\cite{CarCarLos03,CarCarLos08}) exactly computes the \emph{spectral gap} for Maxwell molecules ($\gamma = 0$), and proves that the latter is $N$-uniform ($\lim_{N \to + \infty} c_{N, {\rm sg}} >0$). However, the $L_2((\R^d)^N,\pi_\infty)$-norm used in the spectral gap case, is usually thought to be an unsatisfactory $N$-scalable measure of trend to equilibrium (it is rather associated to the \emph{linearized} kinetic equation for $N=+\infty$, see the last section of~\cite{Vil03} for a longer discussion). By extensivity of entropy, the modified log-Sobolev constant $c_{N,{\rm ls}}$ is believed to be a more reliable quantity.

\emph{Entropy method, $N<+\infty$.} According to~\cite{Vil03}, it is conjectured (and partially proven in the case of Kac's caricature) that the modified log-Sobolev constant of Kac's $N$-particle system, at least for $\nu < 0 $, is of order $c_{N, {\rm ls}} \sim N^{-1 + \gamma/2}$. This is of special interest in the (unphysical) case $\gamma=2$, since it shows that a modified log-Sobolev inequality holds, similarly to the kinetic equation ($N=+\infty$). In general (and in particular for the diffusive, Landau case with Maxwell molecules, or more generally for $ \gamma + \nu^+ \geq 2 $), the large $N$ behavior of the modified log-Sobolev constant of the Kac's particle system is an open problem.

\emph{Coupling method, $N \leq +\infty$.} The use of explicit coupling methods to study the trend to equilibrium of Markov processes (or Markov chains) is now a classical topic on its own, especially for discrete models (see \textit{e.g.} the classical textbook \cite{LevPerWil09}). It is also a well-established topic for continuous models, as well as for non-linear partial differential equations that have an interpretation in terms of a Markovian mechanism. For the granular media equation (diffusive particles interacting through a smooth pairwise potential), and its related $N$-particle system, Markov coupling can give exponential trend to equilibrium, by using a ``strong coupling/coupling creation inequality'' (see for instance \cite{Mal01,BolGenGui12,BolGenGui12bis}, using $CD(c_{N,{\rm cd}},\infty)$-type convexity assumptions on potentials, with $ \lim_{N \to + \infty} c_{N, {\rm cd}} >0$). For the Kac's $N$-particle system of kinetic theory, the only paper known to us quantitatively using a Markov coupling is in~\cite{Oli09}. In the latter, the (almost optimal, and not $N$-uniform) { \bf estimate ($c_{N,{ \rm w2}} \sim 1/(N\ln N)$) is obtained for Kac's caricature}, in accordance with the result cited in~\cite{Vil03}: $c_{N, {\rm ls}} \sim 1/N $.

\emph{Quantitative propagation of chaos.} Finally in~\cite{MisMou11}, the authors have reversed the point of view of Kac's program, and proved indirectly the trend to equilibrium of Kac's $N$-particle system by pulling-back the long time stability of the kinetic (mean-field $N=+\infty$ limit) equation using uniform in time propagation of chaos.

\paragraph{Motivation of the paper.}~

The goal of this paper is to develop on the Kac's particle system with Maxwell molecules a {\bf ``weak approach'' of the (quadratic) coupling method, uniformly in the number of particles $N$}. The latter results extend in spirit the classical paper by~Tanaka \cite{Tan78}, where the quadratic Wasserstein distance between the solution of the kinetic equation with Maxwell collisions and the equlibrium Gaussian distribution (the Mawellian) is shown to be decreasing through time, with a similar coupling argument, but without quantitative analysis. In a sense, the analysis in the present paper makes Tanaka's argument quantitative (with respect to time), and available for the Kac's $N$-particle system.

More precisely, we will obtain power law trends to equilibrium with respect to a permutation symmetrized version of the quadratic Wasserstein distance, and upon estimates on higher moments of the velocity distribution. {\bf Up to our knowledge, this is the first time this type of estimate is obtained directly on the Kac's particle systems}. Such results are similar in spirit to the classical results mentioned above (\cite{Cer82,CarCar92,BobCer99,TosVil99_trend,Vil03}) that are obtained with entropy methods using weakened ``entropy/entropy creation'' inequalities. Yet, the coupling method has some noticeable specificities:
\begin{enumerate}[(i)]
\item{\it Maxwell restriction.} Because it requires coupling of simultaneous collisions, the analysis is restricted to Maxwell molecules ($ \gamma = 0$). 
\item { \emph{Angular condition}} The analysis is independent of the scattering angular distribution $ \nu \in ] - \infty , 2]$.
\item { \emph{Particle system size}} The analysis is independent of the particle system size $N$. It can work similarly for the kinetic equation $N=+\infty$.
\item { \emph{A priori estimates}} The analysis depends only on higher $>2$ moments of velocity distributions, and { \bf not on regularity estimates}. Moments are known to exactly propagates through time (finite and infinite moments remain so, see the classical paper~\cite{IkeTru56}) for the $N=+\infty$ kinetic equation with Maxwell molecules, but unfortunately, not so directly for the Kac's $N$-particle system. We will use the easiest case of order $4$ moments.
\item{ \emph{Constants}} Constants explicitable.
\end{enumerate}

We finally also suggest some negative results in the form of two counterexamples to stronger versions of ``coupling/coupling creation inequalities''. Although similar in spirit, the latter have a different interpretation as compared to the counterexamples to Cercignani's conjecture (see~\cite{Bob88,Vil03}) in the entropy context:  they provide information on the limitation of the specific choice of the coupling, but not directly on the trend to equilibrium of the model. Here are the counterexamples:
\begin{enumerate}[(i)]
\item Velocity distributions with sufficiently heavy tails can make the coupling creation vanish. This first counterexample shows that the obtained ``coupling/coupling creation inequality'' \emph{must} involve some higher order (say, $>2$) velocity distribution moments.
\item There exists a continuous perturbation of the identity coupling at equilibrium for which however the coupling creation is sub-linearly smaller than the coupling itself. This second type of counterexample shows that even with moment restrictions, a \emph{sub-exponential trend is unavoidable}.
\end{enumerate}
As discussed above, the latter facts are consistent with known power law behaviors depending on moment conditions (see again~\cite{CarLu03}) in the angular cut-off case $(\nu <0, \gamma=0)$. However the sub-optimality of the considered coupling may be conjectured in the diffusive Landau case $(\nu = 2, \gamma =0)$, where exponential convergence with a log-Sobolev gap ($c_{\rm ls} > 0$) is known to occur (see again~\cite{DesVil00}).

\paragraph{Summary of results}~

Here, and in the rest of the paper, the following notation is used
\[
\boxed{
\bracket{o\pare{\pN{u},\pN{v},\pN{u}_{\ast },\pN{v}_{\ast}} }_{N} \eqdef \fracd{1}{N^2} \sum_{n_1 , n_2 =1}^{N} o(\pN{u}_{(n_1)},\pN{u}_{(n_1)},\pN{u}_{(n_2)},\pN{v}_{(n_2)}),
}
\]
in order to account for averages over particles of a two-body observable $o: \pare{\R^d\times \R^d}^2 \to \R$. 

The Markov coupling of Kac's conservative $N$-particle system is a Markov process denoted
\begin{equation*}
  t \mapsto (\pN{U}_t,\pN{V}_t)\equiv (U_{t,(1)},V_{t,(1)}, \hdots,U_{t,(N)},V_{t,(N)}) \in \pare{ \R^{d} \times \R^{d} }^N .
\end{equation*}
$\pN{U}_t \in \pare{ \R^{d} }^N$ and $\pN{V}_t \in \pare{ \R^{d} }^N$ both satisfying the same normalized conservation laws: 
\begin{equation}
  \label{eq:cons}
  \begin{cases}
 \dps    \bracket{ V_t }_N = 0 \as \qquad &[\text{centered momenta}] \\[5pt]
\dps \bracket{ \abs{ V_t }^2 }_N = 1 \as. \qquad & [\text{normalized energy}]
  \end{cases}
\end{equation}
Throughout the paper, we will denote the associated probability distributions of the particle system:
\[
\pi_t \eqdef \Law\p{V_t} \in \calP\p{(\R^d)^N}, \quad \pi_{c,t} \eqdef \Law \p{ \pN{U}_t,\pN{V}_t } \in \calP \p{ \pare{ \R^{d} \times \R^d }^N } ,
\]
and assume that $t \mapsto U_t$ is always taken to be distirbuted according to the equilibrium stationary distribution, given by the uniform distribution on the sphere defined by the conservation laws:
\[
\Law\p{U_t} = \pi_\infty = {\rm unif}_{\S^{Nd-d-1}}.
\]

We also assume that the Kac's $N$-particle systems are constructed from a Boltzmann kernel of the type~\eqref{eq:kernel} with { \bf Maxwell molecules } ($\gamma = 0$), and we will use Levy's normalization on scattering angular distribution
\begin{equation}
  \label{eq:Levy}
  \int_{0}^{\pi} \sin^2 \theta \beta(\d \theta ) = 1 .
\end{equation}

The Markov dynamics of the coupled particle system can be described without ambiguity with the related \emph{ master equation }
\begin{equation*}
  \label{eq:master_c}
  \frac{\d}{\d t} \pi_{c,t} = \calL^{\ast}_{c,N} \pi_{c,t},
\end{equation*}
where $*$ refers to duality between measures and test functions. In the above, the dynamics generator is of the form
\begin{equation}
  \label{eq:full_gen_c}
  \calL_{c,N} \eqdef \fracd{1}{2 N} \sum_{1 \leq n \neq m \leq N } L_{c,(n,m)} ,
\end{equation}
where the two-body coupled generator $L_{c,(n,m)}$ is a { \bf coupled } collision operator acting on the particle pair $(n,m)$ for $n \neq m$, and defined when acting on test functions $\psi \in C^\infty_c(\p{ \R^d \times \R^d }^2)$ by:
\begin{equation}
  \label{eq:coupled_gen}
L_c(\psi)(u,v) \eqdef  \int_{ \S^{d-1} \times [0,\pi]}  \pare{ \psi(u',v') - \psi(u,v) } {\rm unif}_{c,\theta}(n_u,n_v, \d n'_u \d n'_v) \,  \beta(\d \theta ).
\end{equation}
In the above, the parallel spherical coupling ${\rm unif}_{c,\theta}$ is precisely defined in point $(iii)$ of Definition~\ref{def:simparcoupl} below. 
\begin{Rem}~

  \begin{enumerate}[(i)]
  \item Due to the lack of smoothness of ${\rm unif}_{c,\theta} \p{n_u,n_v; \d n'_u \d n'_v}$ (there is a singularity on the extremity set $\set{n_u,n_v \in \S^{d-1} \vert n_u=- n_v}$), there is a difficulty to define rigorously the coupled system without angular cut-off:
\[
b_0 = \int_{[0,\pi]} \beta(\d \theta) < + \infty.
\] 
However, the analysis of the present paper is independent of the latter cut-off value $b_0$, and the proofs will be carried out for arbitrary angular singularity, using a continuity argument.
\item By construction, such contracting Markov couplings { \bf cannot satisfy detailed balance } (time symmetry is broken to obtain a contractive map) on the product space~$\L^2(( \R^d \times \R^d)^N)$, so that there is no simple way (at least known to us) to write a dual (to the master equation~\eqref{eq:coupled_gen}) kinetic equation on the density of $\pi_{c,t}$.
  \end{enumerate}  
\end{Rem}

The latter coupling can be defined without ambiguity by requiring { \bf simultaneous collisions, and parallel coupling of each collision }. This is specified using the following set of rules.
\begin{Def}[Simultaneous parallel coupling]\label{def:simparcoupl} The Simultaneous Parallel Coupling between $t \mapsto U_t$ and $t \mapsto V_t$ is obtained by the following set of rules:
\begin{enumerate}[(i)]
\item Collision times and collisional particles are the same (simultaneous collisions), as implied by~\eqref{eq:full_gen_c}.
\item For each collision, the scattering angles  $\theta \in [0,\pi]$ of are the same, as implied by~\eqref{eq:coupled_gen}.
\item For each coupled collision, the post-collisional directions $n'_u \in \S^{d-1}$ and $n'_v \in \S^{d-1}$ are coupled using the elementary rotation along the great circle (the geodesic) of $\S^{d-1}$ joining $n_u$ and $n_v$. The resulting coupled probability is denoted 
\[
{\rm unif}_{c,\theta} \p{n_u,n_v; \d n'_u \d n'_v}.
\]
\end{enumerate}
\end{Def}

Since the post-collisional directions $(n'_u,n'_v)$ are obtained using {\bf parallel coupling on a sphere}, a strictly (under the crucial assumption that { \bf $d \geq 3$}) positively curved manifold, the latter coupling is bound to be almost surely decreasing, in the sense that for any initial condition and $0 \leq t \leq t+h$
\begin{equation*}
  \bracket{ \abs{\pN{U}_{t+h}-\pN{V}_{t+h}}^2}_N \leq \bracket{ \abs{\pN{U}_{t}-\pN{V}_{t}}^2}_N \as.
\end{equation*}

We then compute the \emph{quadratic coupling creation} defined by
\begin{equation*}
  \fracd{\d}{\d t} \E  \bracket{ \abs{\pN{U}_{t}-\pN{V}_{t}}^2}_N = - \E  \p{  \cc \pare{ \pN{U}_t , \pN{V}_t } } \leq 0,
\end{equation*}
where the ``two-body coupling creation'' functional satisfies
\begin{align}\label{eq:coupl_crea}
\boxed{
\cc(u,v) =  \fracd{d-2}{2d-2}  \bracket{ \abs{u-u_\ast}\abs{v-v_\ast} - (u-u_\ast ) \cdot (v-v_\ast) }_N \geq 0,
}
\end{align}
and can be interpreted as a degree of alignement between the velocity difference $v-v_\ast \in \R^d$, and its coupled counterpart $u-u_\ast \in \R^d$.

In order to relate the coupling and the coupling creation, we have introduced in the present paper an original general sharp inequaIity, proved using brute force calculation: for any vectors $\pN{u} \in (\R^d)^N$ and $\pN{v}  \in (\R^d)^N$ both satisfying the normalized conservation laws~\eqref{eq:cons}, it holds
\begin{empheq}[box=\fbox]{align}\label{eq:fund_ineq_uv}
& 1 - \bracket{\pN{u} \cdot \pN{v}}^2_N \leq \min\pare{\kappa_{\bracket{\pN{u}\otimes \pN{u}}_N},\kappa_{\bracket{\pN{v}\otimes \pN{v}}_N} }  \nonumber \\
& \qquad 
\times \bracket{ \abs{\pN{u}-\pN{u}_\ast}^2\abs{\pN{v}-\pN{v}_\ast}^2 - \pare{\pare{\pN{u}-\pN{u}_\ast} \cdot \pare{\pN{v}-\pN{v}_\ast}}^2}_N, \quad \forall \pN{u},\pN{v}\in \S^{Nd-d-1} . 
  \end{empheq}
 In the above, the spectral quantity
\begin{equation}
  \label{eq:kappa}
   \kappa_{S} \eqdef \pare{1- \lambda_{\max}(S)}^{-1} \in [\frac{d}{d-1},+\infty]
\end{equation}
is defined with the spectral radius $\lambda_{\max}(S) \leq 1$ of a positive trace $1$ symmetric matrix. It is finite if and only if $S$ is of rank at least $2$ (non-alignement condition). Note that if $\bracket{\pN{u} \cdot \pN{v}}_N \geq 0$, then the coupling distance satisfies $\bracket{ \abs{ \pN{u} - \pN{v} }^2}_N \leq 2 \p{  1 - \bracket{\pN{u} \cdot \pN{v}}^2_N }$. The equality case in~\eqref{eq:fund_ineq_uv} is achieved (sharpness) under some strong isotropy and co-linear coupling conditions, detailed in Section~\ref{sub:spec}.
\begin{Rem}
  The inequality~\eqref{eq:fund_ineq_uv} can be interpreted as a way to bounde from above the euclidean distance $\bracket{ \abs{ \pN{u} - \pN{v} }^2}_N$ with a quadratic average of {\bf alignement} between the velocity difference $v-v_\ast \in \R^d$ and $u-u_\ast \in \R^d$.
\end{Rem}

It is then of interest to compare that the alignement functional in the right hand side of~\eqref{eq:fund_ineq_uv}, and the coupling creation functional~\eqref{eq:coupl_crea}. They differ by a weight of the form $\abs{u-u_\ast}\abs{v-v_\ast}$ which implies that the strong ``coupling/coupling creation'' constant
\[
 c_{2,N} \eqdef \inf_{ \pN{u},\pN{v}\in \S^{Nd-d-1}   } \frac{\cc \p{ \pN{u},\pN{v} } }{ 2\bracket{\abs{\pN{u}-\pN{v}}^2}_N } 
\]
is degenerated when the number of particles becom large: $\lim_{N \to + \infty} c_{2,N} = 0$ (see the counterexamples of Section~\ref{sec:cex} for more details). However, a direct Hölder inequality yields some weaker power law versions (see details in Section~\ref{sec:results}) for any $\delta >0$,  of the form
\begin{equation}
  \label{eq:rate}
\boxed{
 c(\delta,u,v) \leq \frac{\cc \p{ \pN{u},\pN{v} } }{ 2\bracket{\abs{\pN{u}-\pN{v}}^2}_N^{1+1/2\delta} },
}
\end{equation} 
where the constant $c\p{\delta,\pN{u},\pN{v}}$ can be lower bounded by \emph{($N$-averaged) moments} of the velocity distributions, of any order strictly greater then $2(1 + \delta)$. Additional control on the positive correlation condtion $\bracket{\pN{u} \cdot \pN{v}}_N \geq 0$, and on the isotropy of $\bracket{\pN{u}\otimes \pN{u}}_N$ are required by the inequality~\eqref{eq:fund_ineq_uv}. 

Upon a priori control of such moments, this leads to a { \bf power law trend to equilibrium of the Kac's system distribution}, of order $\limop{\sim}_{t \to + \infty} t^{-\delta}$. The trend to equilibrium is obtained with respect to a permutation symmetrized quadratic Wasserstein distance defined by the quotient distance on $(\R^d)^N /  {\rm Sym }_N$: $d_{\rm sym}(u,v) \eqdef \inf_{\sigma \in {\rm Sym}_N} \pare{\bracket{ \abs{\pN{u}-\pN{v}_{\sigma(\,.\,)} }^2}_N}^{1/2}$. The latter is natural for exchangeable distributions, and necessary to handle the positive correlation assumption $\bracket{ u \cdot v }_N \geq 0 $.

The case of order $4$-moments is finally treated explicitly, and sub-linear trends are estimated. For any $0 < \delta < 1$, we prove that:
\[
\boxed{
d_{{\rm sym},W_2}\p{\pi_t,\pi_\infty}  \leq \p{ d_{{\rm sym},W_2}\p{\pi_0,\pi_\infty}^{-1/\delta} + c_\delta \p{t-t_\ast}^+   }^{-\delta},
}
\]
where the cut-off time satisfies depends logarithmically on the initial order $4$:
\[
t_\ast = 2 \p{\ln \p{ \frac{d}{d+2} \E\bracket{\abs{V_0}^4}_N - 1}  }^+.
\]
and $c_\delta > 0$ is explicitly computable. For instance, we find that
\[
\lim_{\delta \to 1}\lim_{d \to +\infty} \lim_{N\to + \infty} c_{\delta,N} \geq 10^{-3},
\]
which although sub-optimal, is not unreasonably small.

\paragraph{Contents}~

In Section~\ref{sec:not_res}, we recall some notation and basic concepts related to kinetic theory and probabilistic couplings for Markov particle systems. We then detail the results of the present work.

In Section~\ref{sec:coupl}, the parallel, spherical coupling of interest is detailed, together with the precise calculation of the associated quadratic coupling creation.

In Section~\ref{sec:spec}, the special inequality between coupling distance and colinearity of coupled pairs is proven. 

In Section~\ref{sec:proofs}, some details of proofs are given.

\section{Notation and precise results}\label{sec:not_res}

\subsection{Kinetic theory}
As usual, the velocities of a pair of collisional particles are denoted
\[(v,v_{\ast}) \in \R^d\times \R^d,\] 
and the post-collisional quantities are denoted by adding the superscrpit~$'$. All particles are assumed to have the same mass so that the conservation of momentum imposes
\begin{equation*}
  v'+v'_\ast = v+v_\ast ,
\end{equation*}
and conservation of energy imposes
\begin{equation*}
   \abs{v'}^2+\abs{v_\ast'}^2=\abs{v}^2+\abs{v_\ast}^2.
\end{equation*}
As a consequence, the relative speed is also conserved
\begin{equation*}
   \abs{v'-v_\ast'}=\abs{v-v_\ast}.
\end{equation*}
The post-collisional velocities of a particle pair are thus given by the standard collision mapping
\begin{equation}
\label{eq:collision}
    \begin{cases}
      v' = \frac{1}{2}(v+v_\ast) + \frac{1}{2} \abs{v-v_\ast} n'_v, \\[2pt]
      v'_\ast = \frac{1}{2}(v+v_\ast) - \frac{1}{2} \abs{v-v_\ast} n'_v,
    \end{cases}
  \end{equation}
where
\[
(n_v,n'_v) = \pare{\frac{v-v_\ast}{\abs{v-v_\ast}},\frac{v'-v'_\ast}{\abs{v'-v'_\ast}}} \in \S^{d-1} \times \S^{d-1}
\] 
denote the pre-collisional/post-collisional directions. The \emph{scattering or deviation angle} $ \theta \in [0,\pi] $ of the collision is then uniquely defined as the half-line angle between the pre-collisional and the post-collisional directions:
\begin{equation*}
  \label{eq:devangle}
  \cos \theta \eqdef n'_v \cdot n_v.
\end{equation*}
The binary collisions are then specified by the following operator (a Markov generator) acting on test functions $\ph \in C^{\infty}_c((\R^d)^2)$:
\begin{equation}
   \label{eq:gen_levy}
   L(\ph)(v,v_\ast) \eqdef \int_{\S^{d-1} \times [0,\pi]} \pare{\ph(v',v'_\ast) - \ph(v,v_\ast)}  b(v-v_\ast , \d n'_v ) ,
 \end{equation}
where in the above the \emph{Boltzmann collision kernel} $b$ can be decomposed as
\[
b(v-v_\ast , \d n'_v ) \eqdef \int_{\theta \in [0,\pi]} {\rm unif}_{\theta}(n_v, \d n'_v) \, \beta(\abs{v-v_\ast},\d \theta),
\] 
with (i) $\beta(\abs{v-v_\ast}, \d \theta)$ the \emph{angular collisional kernel}, a positive measure on $[0,\pi]$ satisfying the Levy normalization condition~\eqref{eq:Levy}, and (ii) $ {\rm unif}_{\theta}$ is the uniform probability distribution on the sphere of collisional directions $\S^{d-1}$ with prescribed scattering (or deviation) angle $\theta$. More formally: 
\begin{equation}
  \label{eq:rand_rot}
  {\rm unif}_{\theta}(n_v, \d n'_v) \eqdef {\rm unif}_{\set{n'_v \in \S^{d-1} \st n_v \cdot n'_v = \cos \theta }}\pare{ \d n'_v },
\end{equation}
where ${\rm unif}_S$ denotes the uniform probability distribution on a sphere $S$ in euclidean space. The introduction of ${\rm unif}_{\theta}$ will be convenient to describe the coupled collision ${\rm unif}_{c,\theta}$.

By construction, ${\rm unif}_\theta$ satisfies the \emph{detailed balance condition} (micro-reversibility) with invariant probability the uniform distribution on the sphere $\S^{d-1}$. Formally:
\begin{equation*}
  \label{eq:rev}
\d n_v {\rm unif}_{\theta} ( n_v , \d n'_v ) = \d n_v '  {\rm unif}_{\theta} ( n'_v , \d n_v ) \in \calP\pare{\S^{d-1} \times \S^{d-1} },
\end{equation*}
where we implicitly define $\d n_v = {\rm unif}_{\S^{d-1}}( \d n_v )$. It is convenient to keep in mind that~\eqref{eq:rev} extends by measure decomposition to the following version of detailed balance in the euclidean ambient space
\[
\d v \d v_\ast {\rm unif}_{\theta} ( n_v , \d n'_v ) = \d v' \d v'_\ast  {\rm unif}_{\theta} ( n'_v , \d n_v )
\] 
as (unbounded) positive measures in $\R^d \times \R^d$; by a tensorization argument, the latter yields \emph{reversibility} (see below) of the conservative Kac's particle system.

The conservative Kac's $N$-particle system is then defined as a Markov process
\begin{equation}
  \label{eq:part_sys_nocoupled}
 t \mapsto \pN{V}_t \equiv (\pN{V}_{t,(1)}, \hdots,\pN{V}_{t,(N)}) \in \pare{ \R^{d} }^N ,
\end{equation}
whose probability distribution is described without ambiguity with the related master equation
\begin{equation}
  \label{eq:master}
  \frac{\d}{\d t} \pi_t = \calL^{\ast}_N \pi_t,
\end{equation}
holding on the probability distribution flow:
\begin{equation}\label{eq:p_flow}
  \pi_t  \equiv \pi_t^{N}\pare{ \d v_{(1)} \ldots \d v_{(N)} } \eqdef \Law( V_t \equiv \pare{V_{t,(N)},\ldots ,V_{t,(N)}}  ) \in \calP_{\rm sym}((\R^d)^N) \quad t \geq 0,
\end{equation}
where $\calP_{\rm sym}((\R^d)^N)$ denotes permutation symmetric probability distributions (assuming the initial condition $\pi_0$ is already permutation symmetric). Each of the $V_{t,(n)}, 1 \leq n \leq N$ represent the velocity of a physical particle, the whole system of particles being subject to the random binary elastic collisions.

For arbitrary $N \geq 2$, the Markov generator in $(\R^d)^N$ have the following structure:
\begin{equation}
  \label{eq:full_gen}
  \calL_N \eqdef \fracd{1}{2N} \sum_{1 \leq n \neq m \leq N} L_{(n,m)} ,
\end{equation}
where $L_{(n,m)}$ is the Markov generator~\eqref{eq:gen_levy} with state space $ \pare{ \R^{d} \times \R^{d}}^2$, the subscript $(n,m)$ denoting the action on the corresponding pair of particles. The $\frac{1}{N}$ scaling in~\eqref{eq:full_gen} can be physically understood by stating that each individual particle is subject to a collision mechanism with $O(1)$ rate and a uniformly picked other particle. For elastic collisions, we have that $L( \ph)(v,v_\ast = v) = 0$, so that if we consider test functions $\psi \in C^\infty_c((\R^d)^N)$ of average type:
\[
\psi(v) = \bracket{\ph(v)}_N,
\]
we get, thanks to the factor $1/2N$ in~\eqref{eq:full_gen}:
\begin{align*}
  \calL \psi (v) & = \frac{1}{2}  \bracket{ L \p{(v,v_\ast) \mapsto \ph(v) + \ph(v_\ast)}  }_N. \\
 & = \bracket{  L \p{ \ph \otimes \one{} }(v,v_\ast) }_N .
\end{align*}

By construction, the process~\eqref{eq:part_sys_nocoupled} satisfies the physical conservation laws of momentum and energy, that will be taken centered and normalized according to~\eqref{eq:cons} throughout the paper. Moreover, the fundamental detailed balance condition~\eqref{eq:rev} implies detailed balance at the level of the particle system. More precisely:
\begin{enumerate}[(i)]
\item The unique stationary probability distributions is given by the sphere of conservation laws:
\[
\pi_\infty = {\rm unif}_{\S^{d(N-1)-1}} ( \d v_{(1)} \ldots\d v_{(N)} ) \in \calP((\R^d)^N);
\]
where we have implicitly define the unit sphere with normalization condition~\eqref{eq:cons}
\[
\S^{d(N-1)-1} = \set{ v \in (\R^d)^N \st \bracket{v}_N = 0 , \,  \bracket{ \abs{ v }^2 }_N  = 1 }.
\]
\item $\pi_{\infty}$ is in fact an \emph{equilibrium}, in the sense that the process in stationary distribution is \emph{time reversible}
\[
\Law\p{ V_{0}} = \pi_\infty \Rightarrow \Law\p{ V_{t}, 0 \leq t \leq T} = \Law\p{ V_{T-t}, 0 \leq t \leq T} \quad \forall T >0 .
\]
\item Equivalently to~$(ii)$, on has $\calL^{\ast}_N = \calL_N$ in the sense of self-adjointness in the Hilbert space $\L^2\pare{(\R^d)^N, \d v_{(1)} \ldots\d v_{(N)} }$, or alternatively in $\L^2\pare{\S^{d(N-1)-1}}$.
\end{enumerate}

Let us also recall that the latter process can be constructed explicitly in the case of Maxwell molecules ($\beta(\abs{v-v_\ast}, \d \theta) \equiv \beta(\d \theta)$), and angular cut-off ($\int_{[0,\pi]} \beta( \d \theta) <+ \infty $):
\begin{enumerate}[(i)]
\item Each particle perform a collision with a fixed rate $b_0 := \int_{[0,\pi]} \beta( \d \theta) $, and with a uniformly randomly chosen other particle.
\item The scattering angle of each collision is independently sampled according to the probability defined by $\beta(\d \theta )/b_0$.
\item The random post-collisional directions $n'_v$ is uniformly sampled with ${\rm unif}_{\theta}$ and scattering angle prescribed by $(ii)$.
\end{enumerate}
The general angular collisions can then be obtained (rigorously, see~\cite{EthKur85}) as the limit of the latter.

If we denote the probability $\pi_t$ as a (generalized) probability density function with reference measure $\d v_{(1)} \ldots \d v_{(N)}$
\[
\pi_t \equiv f_t( v_{(1)} \ldots v_{(N)} ) \d v_{(1)} \ldots \d v_{(N)},
\]
then the detailed balance conditions yields the usual explicit \emph{dual kinetic equation}, for any $v \in (\R^d)^N$:

\begin{align*}
  \fracd{\d}{\d t} f_t(v) & =\fracd{1}{2 N} \sum_{n,m =1}^{N} \int_{\S^{d-1} \times [0,\pi]} \pare{f_t(v^{'}) - f_t(v)}  b\p{ n_{v_{(n,m)} } , \d n'_{v_{(n,m)}} } \nonumber \\
& = \calL_{N} f_t (v),
\end{align*}
where in the above, the subscript $v_{(n,m)} =(v_{(n)},v_{(m)}) \in \R^d \times \R^d$ refers to the corresponding pair of particles.

Finally, one says that weak propagation of chaos holds, if for any time $t \geq 0$, the marginal distribution of $k$ given particles of the above particle system ($k$ being fixed) is converging (in probability distribution) to a product measure when $N \to + \infty$ (independence). Under this assumption, the large $N$ limit of the \emph{one body} marginal distribution $\pi_t \in \calP(\R^d)$ of the particle system satisfies an evolution equation in closed form with a quadratic non-linearity given by:
\begin{equation}
  \label{eq:non-lin}
  \fracd{\d}{\d t} \int_{\R^d} \ph \, \d \pi_t = \int_{\R^d \times \R^d} L \pare{ \ph \otimes \one{} } \d \pi_t \otimes \d \pi_t,
\end{equation}
where in the above $\ph $ is a test function of $\R^d$. The latter can be easily derived from the master equation~\eqref{eq:master}, by choosing tests functions in $( \R^d)^N$ of the form $\bracket{\ph(v)}_N$ with $\ph \in C^\infty_c(\R^d)$. The dual kinetic equation of the non-linear equation~\eqref{eq:non-lin} is the famous Boltzmann equation in $\R^d$ with Maxwell collision kernel $b$. The usual expression on the one particle velocity density, denoted $\dps f_t(v) \d v \equiv \pi_t( \d v ) $, is then:
\begin{equation}
  \label{eq:boltz}
  \fracd{\d}{\d t} f_t(v) = \int_{\R^d \times \S^{d-1} } \pare{f_t(v')f_t(v'_\ast) - f_t(v)f_t(v_\ast)} \d v_\ast \, b(n_v , \d n'_{v}) .
\end{equation}

\subsection{Coupling}

Let $(E,d)$ denote a Polish state space ($E :=\p{\R^d}^N$ euclidean with $d(u,v) := \bracket{\abs{u-v}^2}_N $ in the present paper). We say that a time-homogenous Markov process in the product space
\[
t \mapsto (U_t,V_t) \in E \times E,
\]
is a \emph{Markov coupling}, if the marginal probability distribution of the two processes $t \mapsto U_t \in E $ and $t \mapsto V_t \in E $ are two instances of the same Markov dynamics, with possibly different initial distributions. We will be interested in weakly contracting couplings, where the coupling distance is almost surely decreasing:

\begin{equation}
  \label{eq:coupl_mon}
 d(U_{t+h},V_{t+h}) \leq d(U_{t},V_{t}) \as, \quad \forall t,h \geq 0,
\end{equation}
and will especially consider \emph{quadratic coupling creation} defined by:
\begin{equation*}
  \label{}
\cc(u,v) \eqdef - \left . \frac{\d}{ \d t} \right |_{t=0} \E_{(U_0,V_0) = (u,v)} \pare{ d(U_{t},V_{t})^2 } \geq 0.
\end{equation*}
From an analytic point of view, if $\calL_c$ denotes the Markov generator of the coupled process $t \mapsto (U_t,V_t) $, and $\calL$ the Markov generator of the marginal process $t \mapsto U_t$ (or $t \mapsto V_t$) it is useful to keep in mind that:

\begin{enumerate}[(i)]
\item Coupling amounts to consider the compatibility condition: for any $(u,v) \in E \times E$, and $\ph$ a test function:
\begin{equation*}
  \label{}
\calL_c(\ph \otimes \one{} )(u,v) = \calL(\ph)(u), \quad  \calL_c(\one{} \otimes \ph )(u,v) = \calL(\ph)(v).
\end{equation*}
\item Couplings generators $\calL_c$ invariant by permutation of the role of the two variables $(u,v) \in E^2$ are called ``symmetric'' (if $\ph(u,v)=\ph(v,u)$, then $\calL_c(\ph)(u,v) = \calL_c(\ph)(v,u)$). We will only use symmetric couplings, although this fact is unimportant in the analysis.
\item The quadratic coupling creation functional can be defined as
\begin{equation}
  \label{eq:cc_def}
\cc(u,v) \eqdef   - \calL_c\pare{d^2}(u,v).
\end{equation}
\end{enumerate}
In the most favorable situation, one can expect a contractive coupling / coupling creation inequality with constant $0 < c_{2} < + \infty$:
\begin{equation}
  \label{eq:cc_contr}
  d(u,v)^2  \leq \frac{1}{2c_{2} }  \cc(u,v) , \quad \forall (u,v) \in E^2.
\end{equation}
The latter leads to contractivity with respect to the quadratic Wasserstein distance ($c_2 \leq c_{w_2}$ with the notation of the introduction):
\begin{equation}
  \label{eq:wass_contr}
  d_{W_2}\pare{\pi_{1,t}, \pi_{2,t}} \leq d_{W_2}\pare{\pi_{1,0}, \pi_{2,0}} {\rm e}^{- c_{2} t },
\end{equation}
where in the above $t \mapsto (\pi_{1,t}, \pi_{2,t}) $ are two probability flows solution of the master equation
\begin{equation*}
  \frac{\d}{ \d t} \pi_t = \calL^\ast \pi_t,
\end{equation*}
and the quadratic Wasserstein distance is as usual defined by
\begin{equation}
  \label{eq:wass_def}
  d_{W_2}\pare{\pi_{1}, \pi_{2}} \eqdef \inf_{\pi \in \Pi(\pi_{1}, \pi_{2})} \pare{ \int_{E^2} d(u,v)^2 \pi(\d u , \d v) }^{1/2},
\end{equation}
 $\Pi(\pi_{1}, \pi_{2})$ denoting the set of all possible couplings with marginal distributions $\pi_1$ and $\pi_2$. Finally, if $\pi_\infty \in \calP(E)$ denotes a stationary probability distribution for $\calL$, then~\eqref{eq:wass_contr} yields exponential convergence of the flow $t \mapsto \pi_t$ towards $\pi_\infty$ with respect to Wasserstein distance.

In the context of the present paper, contractivity estimates as~\eqref{eq:wass_contr} are too strong too hold.  We will seek for a \emph{power law trend to equilibrium} in the form
\[
\frac{\d^+}{\d t} d_{W_2}\p{\pi_t,\pi_\infty} \leq - c_\delta \p{\pi_t} d_{W_2}\p{\pi_t,\pi_\infty}^{1+1/\delta},
\]
or equivalently
\[
d_{W_2}\p{\pi_t,\pi_\infty}  \leq \p{ d_{W_2}\p{\pi_0,\pi_\infty}^{-1/\delta} + \frac{1}{\delta}\int_{0}^t c_\delta \p{\pi_s} \d s   }^{-\delta}.
\]
where we denote
\[
\frac{\d^+}{\d t} x \eqdef \liminf_{h \to 0^+} \frac{x_{t+h} - x_t }{h}.
\]

\begin{Rem}~
  \begin{enumerate}[(i)]
  \item The limit $\delta \to + \infty$ gives back the exponential trend~\eqref{eq:wass_contr}.
\item The present paper will compute precise estimates of $c_\delta(\pi)$ in terms of moments of $\pi$ of order $2q(1+\delta)$, for any $q >1$.
  \end{enumerate}
\end{Rem}

Consider now the case of an exchangeable (particle permutation symmetric) $N$-particle system as a random vector $U \in (\R^d)^N$, where $\R^d$ is euclidean. Strictly speaking, the state space is obtained by quotienting out the symmetric group ${\rm Sym}_N$:
\[
E := (\R^d)^N / {\rm Sym}_N,
\]
or equivalently considering the subset of empirical distributions:
\[
E := \calP_N(\R^d) = \set{ \pi \in \calP(\R^d) \st \exists u \in  (\R^d)^N, \quad \pi =    \frac{1}{N} \sum_{n=1}^{N} \delta_{u_{(n)}} }.
\]
The former can be endowed with the associated orbifold distance
\[
d_{\rm sym}(u,v) \eqdef \inf_{\sigma \in {\rm Sym}_N} \pare{\bracket{ \abs{\pN{u}-\pN{v}_{\sigma(\,.\,)} }^2}_N}^{1/2} ,
\]
which is by definition equivalently the quadratic Wasserstein distance induced by $\calP(\R^d)$:
\[
d_{\rm sym}(u,v) = d_{W_2, \calP_N} \pare{ \frac{1}{N} \sum_{n=1}^{N} \delta_{u_{(n)}} , \frac{1}{N} \sum_{n=1}^{N} \delta_{v_{(n)}}  }.
\]
This leads to the following definition.
\begin{Def}
Let $\calP_{\rm sym}((\R^d)^N) \simeq \calP((\R^d)^N / {\rm Sym}_N)$ the set of symmetric (exchangeable) probabilities of $(\R^d)^N$. The ``two-step'' or ``symmetric'' quadratic Wasserstein distance on $\calP_{\rm sym}((\R^d)^N)$ denoted $d_{W_2, {\rm sym}}$ is defined as the usual quadratic Wasserstein distance~\eqref{eq:wass_def} on the quotient space $(\R^d)^N / {\rm Sym}_N$ endowed with the distance $d_{\rm sym}$.\end{Def}

However, in the present paper, the Markov couplings of two particle systems in $(\R^d)^N$ \emph{won't} be constructed on the product space
\[
\pare{ \R^{d} }^N \! \!  / {\rm Sym}_N \times \pare{ \R^{d} }^N \! \!  / {\rm Sym}_N,
\]
but on the non-quotiented space $\pare{ \R^{d} \times \R^{d} }^N$. The generator $\calL_{c,N}$ of the coupled system conserve permutation invariance only \emph{globally}, and the exchangeability of particle will be broken at the initial coupling (see the proof in Section~\eqref{sub:proof_main}). This corresponds to the intuitive picture of pairing particles of two exchangeable sets once and for all. 

More precisely, we will use the symmetrized Wasserstein distance by picking an initial condition as follows
 \begin{Lem}\label{lem:couplrep}
   Let $\pi_1,\pi_2 \in \calP_{\rm sym} \p{(\R^d)^N}$ be two exchangeable probabilities. Then there exists a random variable representation $(U_0,V_0, \Sigma) \in (\R^d \times \R^d)^N \times {\rm Sym}_N$ with $\Law(U_0)=\pi_1$, $\Law({V_0}) = \pi_2$ such that:
   \begin{enumerate}[(i)]
   \item $d_{W_2, {\rm sym}}(\pi_1,\pi_2)^2 = \E \bracket{ \abs{ U_0-V_{0,\Sigma \p{.}} } ^2}_N $.
   \item If $V_0$ is almost surely centered ($\bracket{V_0}_N=0 \as $), then $\bracket{ U_0 \cdot V_{0, \Sigma \p{.} } }_N \geq 0 \as $.
   \end{enumerate}
 \end{Lem}
 \begin{proof}
  First, $(i)$. Since ${\rm Sym}_N$ is finite, $\p{ \p{\R^{d} }^N / {\rm Sym_N}, d_{\rm sym}}$ is a Polish quotient metric space, so that existence of an optimal coupling is known to hold (see \cite{Vil08}). Then an exchangeable representative $V_0$ of the quotient can be picked uniformly at random, and then $\Sigma $ can be defined such that:
\[
\bracket{ \abs{ U-V_{\Sigma \p{.}} }^2}_N =  \inf_{\sigma \in {\rm Sym}_N} \bracket{ \abs{\pN{U}_0-\pN{V}_{0,\sigma(\,.\,)} }^2}_N   \as.
\]

Second, $(ii)$.  By the centering assumption on $V_0$:
\[
\frac{1}{ N ! } \sum_{\sigma \in {\rm Sym}_N} \bracket{  \pN{U}_0 \cdot \pN{V}_{0,\sigma(\,.\,)} }_N = 0,
\]
and the result follows by definition of quotient metric $d_{\rm sym}$.
 \end{proof}

\subsection{Results}\label{sec:results}
We can now detail the results of the present paper.

We first give the special inequality that will enable to derive coupling / coupling creation inequalities.

\begin{The}\label{the:fund_ineq}
Denote $ \kappa_{S}\eqdef \pare{1- \lambda_{\rm max}(S)}^{-1} \in [d/(d-1),+\infty] $ where $\lambda_{\rm max}(S)$ is the maximal eigenvalue of a trace~$1$ symmetric positive matrix $S$. Let $(U,V) \in \R^d \times \R^d$ be a couple of centered and normalized (with $\E \abs{U}^2 =\E \abs{V}^2 =1$) random variables in euclidean space. Let $(U_\ast,V_\ast)\in \R^d \times \R^d$ be an i.i.d. copy. Then the following inequality holds:
\begin{align}\label{eq:fund_ineq}
&1-\E \p{U \cdot V}^2  \nonumber \\
&\hspace{1cm} \leq \min\pare{ \kappa_{\E\pare{U\otimes U}},  \kappa_{\E\pare{V\otimes V}} }  \E \pare{ \abs{U-U_\ast}^2\abs{V-V_\ast}^2 - \pare{\pare{U-U_\ast} \cdot \pare{V-V_\ast}}^2}.
\end{align}
Note that $ \min\pare{ \kappa_{\E\pare{U\otimes U}},  \kappa_{\E\pare{V\otimes V}} } < + \infty $ if and only if either $E\pare{U\otimes U}$ or $E\pare{V \otimes V}$ have rank at least $2$ ({\it i.e.} are not degenerate on a line).

Moreover, a sufficient condition for the equality case in~\eqref{eq:fund_ineq} is given by the following isotropy and co-linear coupling conditions
\begin{enumerate}[(i)]
\item $
\frac{U}{\abs{U}} = \frac{V}{\abs{V}} \quad \as.
$
\item Either $\E\pare{U\otimes V} = \E\pare{U\otimes U}=\frac{1}{d} \Id$ or $\E\pare{U\otimes V} = \E\pare{V\otimes V}=\frac{1}{d} \Id$. 
\end{enumerate}
\end{The}

\begin{Rem}~
  \begin{enumerate}[(i)]
  \item If $\E( U \cdot V ) \geq 0$ (positive correlation condition), then 
    \begin{equation*}
      \frac12 \E \p{ \abs{U-V}^2 } \leq 1-\E \p{U \cdot V}^2.
    \end{equation*}
\item Inequality~\eqref{eq:fund_ineq} controls the averaged square \emph{coupling distance} $\abs{U-V}$ with the average \emph{parallelogram area} spanned by the pair $(U-U_\ast,V-V_\ast)$.
\item The key point to obtain Theorem~\ref{th:main} below is to apply inequality~\eqref{eq:fund_ineq} using the probability space $(\Omega, \P) \equiv ([1,N],\bracket{\, . \, }_{N}$). This yields an inequality of the form~\eqref{eq:rate}.
  \end{enumerate}
\end{Rem}
We thus obtain the main theorem on the power law trend to equilibrium with respect to the quadratic Wassertsein distance:
\begin{The}\label{th:main}
  Let $t \mapsto \pN{V}_t \in \pare{\R^d}^N$ any Kac's conservative particle system with Maxwell molecules and normalization conditions~\eqref{eq:cons}-\eqref{eq:Levy}. Denote~$\pi_t \eqdef \Law(V_t)$. For any $\delta >0, q > 1$, the following trend to equilibrium holds:
 \begin{equation*}\label{eq:ineq_particle_opt}
\fracd{\d^+}{\d t} d_{\calW_2, {\rm sym} }(\pi_t, \pi_\infty ) \leq -  c_{\delta,q,N}\p{\pi_t} d_{\calW_2, {\rm sym} }(\pi_t, \pi_\infty)^{1+ 1/\delta} ,
  \end{equation*}
where in the above
\begin{align*}
 c_{\delta,q,N}\p{\pi_t} = k_{\delta,q,N} \E\p{ \bracket{ \abs{ V_t }^{2q(1+\delta)} } }^{-1/2q\delta } > 0,
\end{align*}
with $ k_{\delta,q,N} $ a numerical constant (independent of the initial condition and of the angular kernel) satisfying $\liminf_{N \to + \infty} k_{\delta,q,N}  >0 $, and explicitly bounded below (NB: $k_{\delta,q,N} \to 0$ polynomially when $q \to 1$).

\end{The}
The moment can be explicitly estimated, uniformly in $N$, in the case of order~$4$ moments.
\begin{Pro}\label{pro:order4}
  Consider the case $0 < \delta < 1$, $2q(1+\delta) = 4$, in Theorem~\ref{th:main}. We have the lower bound estimate:
\[
d_{W_2}\p{\pi_t,\pi_\infty}  \leq \p{ d_{W_2}\p{\pi_0,\pi_\infty}^{-1/\delta} + c_{N,\delta} \p{t-t_\ast}^+   }^{-\delta}.
\]
where the cut-off time depends logarithmically on the initial radial order $4$ moment and is defined by:
\[
t_\ast = 2 \p{\ln \p{ \frac{d}{d+2} \E\bracket{\abs{V_0}^4}_N - 1}  }^+.
\]
and $c_{\delta,N}$ is a numerical constant (independent of the initial condition and of the angular kernel) satisfying $\liminf_{N \to + \infty } c_{N,\delta} >0$ and explicitly bounded below\footnote{NB: for instance, we found $\lim_{\delta \to 1}\lim_{d \to +\infty} \lim_{N\to + \infty} c_{\delta,N} > 10^{-3}$.}.
\end{Pro}

\begin{Rem}
  Finally, note that similar results as Theorem~\ref{th:main} and Proposition~\ref{pro:order4}, with the same constants, can be obtained directly on the associated non-linear kinetic equation. The sketch of proof is the following. 
  \begin{enumerate}[(i)]
  \item Construct a coupled non-linear equation in $\calP(\R^d \times \R^d)$ with kernel defined from ${\rm unif}_{c,\theta}$, and under angular cut-off (using, say, total variation distance). 
\item Prove the analogue of Theorem~\ref{th:main} using the usual Wasserstein $d_{W_2}$ in $\R^d$.
\item Prove the continuity (in Wasserstein distance) of the solution non-coupled kinetic equation with respect to the angular cut-off parameter (this is done in Section~$5$ of~\cite{TosVil99}, and typically requires an appropriate uniqueness theory).
\item Prove uniform (with respect to the angular cut-off parameter) control on higher moments. For Maxwell molecules, explicit computations of the latter can be carried out (see the classical paper~\cite{IkeTru56}). 
  \end{enumerate}
The details are left for future work.
\end{Rem}

\section{Simultaneous Parallel Coupling}\label{sec:coupl}

\subsection{Parallel coupling of collisions}\label{sub:coupl}
A coupled collision can then be described by expressing the post-collisional velocities
$ (u',u'_\ast,v',v'_\ast) \in \R^{ 2d} \times \R^{ 2d } $
using coupled collision parameters. It is sufficient in order to obtain the above coupling to express, using the \emph{same} collision random parameters, the collision and post-collisional directions $(n_u, n'_u,n_v,n'_v) \in \pare{\S^{d-1}}^2 \times \pare{\S^{d-1}}^2$. This is done using parallel coupling. We state without proof (the reader may resort to a simple drawing here) two equivalent elementary descriptions of the parallel coupling on the sphere. The latter \emph{parallel spherical coupling}.

\begin{Def}\label{def:coupling}
  Let $(n_u,n_v) \in \pare{ \S^{d-1} }^2$ satisfying $n_u \neq -n_v$. There is a unique rotation of $\R^{d}$ denoted 
 \begin{equation*}
   n'_u \mapsto n'_v  = {\rm coupl}_{n_u,n_v}(n'_u) \in \S^{d-1},
 \end{equation*}
called \emph{spherical parallel coupling}, satisfying $n'_u=n'_v$ if $n_u=n_v$, and equivalently defined as follows for $n_u \neq n_v$.
\begin{enumerate}[(i)]
\item $n'_v$ is obtained from $n'_u$ by performing the elementary rotation in $\Span(n_u,n_v)$ bringing  $n_u$ to $n_v$.
\item Denote by $t_u$ a tangent vector of $\S^{d-1}$ at base point $n_u$ of a geodesic of length $\theta$ bringing $n_u$ to $n'_u$. Generate $t_v$ from $t_u$ by using parallel transport in $\Span(n_u,n_v)$ from base point $n_u$ to base point $n_v$. Generate $n'_v$ as the endpoint of the geodesic of length $\theta$ and tangent to $t_v$ at base point $n_v$.
\end{enumerate}
Moreover, it satisfies by construction the symmetry condition
\begin{equation}
  \label{eq:sym}
   {\rm coupl}_{n_v,n_u} =  {\rm coupl}_{n_u,n_v}^{-1} .
\end{equation}
If $n_u = -n_v$ and $\sigma \in \S^{d-1}$, then we will denote by ${\rm coupl}_{n_u,n_v}^{\sigma}$ the unique rotation of $\R^{d}$ satisying $(i)-(ii)$ above, but with an elementary rotation, or a geodesic taken in the plane $\Span(n_u, \sigma) = \Span(n_v, \sigma) $.
\end{Def}

It is necessary to keep in mind that the full mapping $(n_u,n_v) \mapsto  {\rm coupl}_{n_u,n_v}$ is smooth, except at a singularity on the extremity set $\set{n_u,n_v \in \S^{d-1} \vert n_u=- n_v}$. This fact has already been pointed out (\cite{Tan78,FouMou09,FouMis13}) in slightly different contexts, and causes difficulty in order to define uniquely regular Levy generators and associated kinetic non-linear equations. However, we will avoid such technical issues by considering angular cut-off, and we will consider coupled Levy or diffusive generators only at the formal level.

Anyway, it is possible to define a coupled probability transition by randomly generating the coupling geodesic when $n_u = - n_v$. This yields:
\begin{Def}
The \emph{spherical parallel coupling} of a random collision with deviation angle $\theta \in [0,\pi]$ is defined by the following probability transition on $\S^{d-1} \times \S^{d-1}$:
\begin{align}
{\rm unif}_{c,\theta}(n_u,n_v ,\d n'_u \d n'_v ) & \eqdef \p{ \one{n_u \neq -n_v} \delta_{ {\rm coupl}_{n_u,n_v}(n'_u) }( \d n'_v) +\one{n_u = -n_v} \delta_{ {\rm coupl}^{\sigma}_{n_u,n_v}(n'_u) }( \d n'_v) {\rm unif}(\d \sigma) } \,   \nonumber \\
& \qquad \times {\rm unif}_{\theta}\pare{ n_u, \d n'_u } , \label{eq:c_coupl}
\end{align}
\end{Def}
\begin{Lem}
The probability transition~\eqref{eq:c_coupl} verifies the symmetry condition
\begin{align}\label{eq:sym2}
 {\rm unif}_{c,\theta}(n_u,n_v , \d n'_u \d n'_v ) =  {\rm unif}_{c,\theta}(n_v , n_u, \d n'_v \d n'_u ).
\end{align}
It is thus a symmetric Markov coupling of the uniform probability transition ${\rm unif}_{\theta}(n , \d n')$.
\end{Lem}
\begin{proof}
  By construction, ${\rm coupl}_{n_u,n_v}$ and ${\rm coupl}_{n_u,\sigma}$ are isometries. On the other hand, by isotropy, for any isometry $R$ and vector $n_v \in \S^{d-1 }$ we have $R^{-1} {\rm unif}_{\theta}( R n_v, .)= {\rm unif}_{\theta}( n_v, .)$. Finally, the symmetry condition~\eqref{eq:sym} yields~\eqref{eq:sym2}.
\end{proof}

\subsection{Coupled spherical coordinates}

We give a special description of the isotropic probability transition with scattering angle $\theta$.
\begin{Lem}\label{lem:azim}
  Let $\theta \in [0,\pi]$ be given, as well as $ (n_v,m_v)$ two orthonormal vectors in $\S^{d-1}$. Consider the spherical change of variable
\begin{equation}
  \label{eq:azim}
  n'_v= \cos \theta \, n_v + \sin \theta \cos \ph \, m_v + \sin \theta \sin \ph \,  l \in \S^{d-1}
\end{equation}
where $\ph \in [0,\pi]$ is an \emph{azimuthal angle} and $l \in \S^{d-1}$ is such that $(n_v,m_v,l)$ is an orthonormal triplet. Then the image by the transformation~\eqref{eq:azim} of the probability distribution
\begin{equation}
  \label{eq:spher_vol}
  \sin^{d-3} \ph \, \fracd{\d \ph}{ w_{d-3}} {\rm Unif}_{(n_v,m_v)^{\perp} \cap  \, \S^{d-1} }(\d l),
\end{equation}
is the isotropic probability transition ${\rm unif}_{\theta}(n_v, \d n'_v)$ with initial state $n_v$ and scattering angle $\theta$ ($w_{d-3}$ denotes the Wallis integral normalization). In particular, the latter does not depend on the choice of $m_v$.
\end{Lem}
\begin{proof}
${\rm unif}_{\theta}(n_v, \d n_v')$ is defined as the uniform distribution induced by the euclidean structure on the submanifold of $\S^{d-1}$ defined by $n'_v \cdot n_v =\cos \theta$. Moreover the expression of volume elements in (hyper)spherical coordinates implies that for any $m_v \in \S^{d-1}$, the vector $\cos \ph \, m_v + \sin \ph \,  l \in \S^{d-1} $ is distributed (under~\eqref{eq:spher_vol}) uniformly in the $d-2$-dimensional sphere $n_v^{\perp} \cap  \S^{d-1}$. The result follows.
\end{proof}
Of course in the above, only the scattering angle $\theta$ has an intrinsic physical meaning, the azimuthal angle $\ph$ being dependent of the arbitrary choice of the pair $(m_v,l)$. This leads to the core analysis of a spherical coupling.
\begin{Lem}\label{lem:coupl}
  Let $(n_u,n_v) \in \S^{d-1} \times \S^{d-1}$ be given. A pair $(n'_u,n'_v) \in \S^{d-1} \times \S^{d-1}$ is spherically coupled (the spherical coupling mapping is defined in Definition~\ref{def:coupling}), in the sense that $n'_v={\rm coupl}_{n_u,n_v}(n'_u)$ if $n_u\neq n_v$ and $n'_v={\rm coupl}_{n_u,\sigma}(n'_u)$ for some $\sigma \in \S^{d-1}$ otherwise, if and only if
\begin{equation}
  \label{eq:sph_coupl}
  \begin{cases}
    n'_u= \cos \theta \, n_u + \sin \theta \cos \ph \, m_u + \sin \theta \sin \ph \,  l ,\\
    n'_v= \cos \theta \, n_v + \sin \theta \cos \ph \, m_v + \sin \theta \sin \ph \, l,
  \end{cases}
\end{equation}
where in the above $(n_u,m_u,l)$ and $(n_v,m_v,l)$ are both orthonormal sets of vectors such that $(n_u,m_u)$ and $(n_v,m_v)$ belong to the same plane have the same orientation with respect to $l$. Note that if $n_u\neq n_v$, the pair $(m_u,m_v)$ and the angle $\ph$ are defined uniquely up to a common involution (a change of sign of the vectors and the reflexion $\ph \to \pi - \ph$).
\end{Lem}
\begin{proof}
Assume $n_u\neq n_v$.  Denote by $R_\theta$ the unique elementary rotation bringing $n_u$ to $n_v$. By construction $R_\theta m_u = m_v$, and $R_\theta l = l$
\end{proof}
 This immediately implies that the coupled probability transition ${\rm unif}_{c,\theta}(n_u,n_v,\d n'_u \d n'_v)$ is the image using the mapping~\eqref{eq:sph_coupl} above of the uniform probability described in $(\ph,l)$-variables and given by~\eqref{eq:spher_vol}.
\begin{Lem}\label{lem:coupl_2}
Let $(n_u,n_v) \in \S^{d-1} \times \S^{d-1}$ be given. If $n_u = -n_v $ pick $(m_u,m_v)$ in the plane $\Span(n_u,n_v,\sigma)$ for some $\sigma \in \S^{d-1}$. Then the image under the mapping~\eqref{eq:sph_coupl} of the probability distribution
\[
\sin^{d-3} \ph \, \fracd{\d \ph}{ w_{d-3}} {\rm unif}_{(n_v,m_v)^{\perp} \cap  \, \S^{d-1} }(\d l) {\rm unif} (\d \sigma) ,
\]
 is given by ${\rm unif}_{c,\theta}(n_u,n_v,\d n'_u \d n'_v)$.
\end{Lem}

\subsection{Contractivity of spherical couplings}\label{sec:contr}


We can first calculate the quadratic contractivity (''coupling creation'') equation satisfied by parallel spherical couplings.
  
\begin{Lem}\label{lem:coupl_coll}
Consider coupled collisional and post-collisional velocities $(u,u_\ast,v,v_\ast) \in \R^{2d}\times \R^{2d}$, and a parallel spherical coupling using the coordinate expression of Lemma~\ref{lem:coupl}. Then we have:
\begin{align}
& \abs{u'-v'}^2 +\abs{u'_\ast-v'_\ast}^2-\abs{u-v}^2-\abs{u_\ast-v_\ast}^2 = \nonumber\\
& \hspace{1cm} - \sin^2\theta \sin^2\ph 
\pare{\abs{u-u_\ast}\abs{v-v_\ast} - (u-u_\ast)  \cdot (v-v_\ast) } \leq 0 . \label{eq:m2var}
\end{align} 
\end{Lem}
\begin{proof}
We use the following change of variable:
\begin{equation*}
  \begin{cases}
    \dps  s_v \eqdef \fracd{1}{2}\pare{v+v_\ast} \\[2pt]
    \dps d_v \eqdef \fracd{1}{2}\pare{v-v_\ast}
  \end{cases}
\Leftrightarrow
\begin{cases}
    \dps  v  = s_v + d_v \\[2pt]
    \dps v _\ast = s_v - d_v
  \end{cases}.
\end{equation*}
First remark that
\begin{align}
  \abs{u-v}^2  + \abs{u_\ast-v_\ast}^2 & = \abs{s_u - s_v + d_u -d_v}^2 + \abs{s_u - s_v - d_u + d_v}^2 \nonumber \\
& = 2 \abs{s_u - s_v}^2 + 2 \abs{d_u - d_v}^2 \label{eq:c2}
\end{align}
Developing the left hand side of~\eqref{eq:m2var}, and using the conservation laws ($s'=s$ and $\abs{d'}=\abs{d}$), we obtain
\begin{align*}
&  \abs{u'-v'}^2 +\abs{u'_\ast-v'_\ast}^2-\abs{u-v}^2-\abs{u_\ast-v_\ast}^2 \nonumber \\
&\quad = 2 \abs{d_u' -d_v'}^2 - 2 \abs{d_u -d_v}^2 \nonumber \\
& \quad =  - (u'-u'_\ast) \cdot (v '- v'_\ast) + (u-u_\ast) \cdot (v - v_\ast) \nonumber \\
& \quad = - \abs{u-u_\ast}\abs{v - v_\ast} \pare{n'_u \cdot n'_v - n_u \cdot n_v }\\
\end{align*}

Next, we expand $n_u' . n_v'$ using~\eqref{eq:sph_coupl} and obtain:
\begin{equation*}
n_u' . n_v' =  \pare{\cos \theta \, n_u + \sin \theta \cos \ph \, m_u}.\pare{\cos \theta \, n_v + \sin \theta \cos \ph \, m_v }+ \sin^2 \theta \sin^2 \ph.
  \end{equation*}
Next by construction, $(m_u,m_v)$ is obtained from a $\fracd{\pi}{2}$-rotation of $(n_u,n_v)$, so that $n_u . n_v = m_u . m_v$ and  $n_u.m_v=-m_u.n_v$ and
\begin{equation*}
n_u' . n_v' =  \pare{ \cos^2 \theta + \sin^2 \theta \cos^2 \ph } \, n_u.n_v  + \sin^2 \theta \sin^2 \ph .
  \end{equation*}
Using $ 1=  \cos^2 \theta + \sin^2 \theta \cos^2 \ph  + \sin^2 \theta \sin^2 \ph$ we obtain
\begin{equation*}
 n_u' . n_v' - n_u.n_v =  -  \sin^2 \theta \sin^2 \ph \pare{  n_u.n_v  -1 },
  \end{equation*}
and the result follows.
\end{proof}

We can finally  compute the coupling creation functional for the coupled Kac's particle system.
\begin{Lem}
Consider the coupled Kac's particle particle system as defined by~\eqref{eq:full_gen_c}-\eqref{eq:coupled_gen}. Then the coupling creation functional $\cc $ is given by~\eqref{eq:coupl_crea}.
\end{Lem}
\begin{proof}
By definition
\begin{align*}
  \cc(u,v)  & = \frac{1}{2} \bracket{  L_c \pare{(u,v,u_\ast,v_\ast) \mapsto \abs{u-v}^2+\abs{u_\ast-v_\ast}^2} }_N  \\
& =  - \frac{1}{2} \int_{[0,\pi]^2} \sin^2\theta \sin^2\ph  \, \beta(\d \theta)\,  \fracd{\sin^{d-3} \ph\d \ph }{w_{d-3}} \bracket{ \abs{u-u_\ast}\abs{v-v_\ast} - (u-u_\ast) \cdot (v-v_\ast) }_N
\end{align*}
and the result follows from the well known formula for Wallis integrals $w_{d-1}/w_{d-3} = (d-2)/(d-1)$.
\end{proof}

\section{The special inequality and its sharpness}\label{sec:spec}

\subsection{The special inequality}\label{sub:spec}
If $(Z_1,Z_2)$ are two centered random vectors, we will use the notation
\begin{equation*}
  C_{Z_1,Z_2} \eqdef \E \pare{Z_1 \otimes Z_2} .
\end{equation*}

Let $(U,V)$ be two centered random vectors in a Euclidean space, and $(U_\ast,V_\ast)$ an i.i.d. copy. The goal is to bound from above the quadratic \emph{coupling distance} $\E\pare{ \abs{U-V}^2 }$ with the following alignement average
\begin{equation*}
  \E\pare{ \abs{U-U_\ast}^2\abs{V-V_\ast}^2 - \pare{ \pare{U-U_\ast} \cdot \pare{V-V_\ast}}^2  }.
\end{equation*}
The latter can be interpreted as the average parallelogram squared area spanned by the two vector differences $U-U_\ast$ and $V-V_\ast$.

The main computation is based on a brute force expansion and a general trace inequality applied to the full covariance
\begin{equation*}
  C_{(U,V),(U,V)} = \bmat C_{U,U} & C_{U,V} \\ 
C_{V,U} & C_{V,V} \emat.
\end{equation*}

The trace inequality is detailed in the following lemma.
\begin{Lem}\label{lem:trace}
  Let $U$ and $V$ be two centered random vectors in $\R^d$. Then we have 
\[
\Tr\pare{C_{U,U} C_{V,V} } - \Tr \pare{C_{U,V} C_{V,U} }  \leq \min\pare{ \fracd{\lambda_{\rm max} \p{C_{U,U}}}{\Tr \pare{ C_{U,U} } },\fracd{\lambda_{\rm max} \p{C_{V,V}}}{\Tr \pare{ C_{V,V} }}} \pare{\Tr\pare{C_{U,U} } \Tr\pare{ C_{V,V} } - \Tr \pare{C_{U,V} }^2}, 
\]
where $\lambda_{\rm max} \p{.}$ denotes the spectral radius (maximal eigenvalue) of a symmetric non-negative operator. Moreover, the equality case holds if (sufficient condition) either $C_{U,U}$ and $C_{U,V}$ or $C_{V,V}$ and $C_{U,V}$ are co-linear to the identity matrix.
\end{Lem}
\begin{proof}
First assume that $C_{U,U}$ has only strictly positive eigenvalues. In an orthonormal basis where $C_{U,U}$ is diagonal, we have the expression, for $i,j,k \in [\! [1,d]\!]$:
\begin{align*}
  \Tr\pare{C_{U,U} C_{V,V}-C_{U,V} C_{V,U} }    & =   \sum_{i}  C^{i,i}_{U,U} C^{i,i}_{V,V} -  \sum_{j,k}   C^{j,k}_{U,V} C^{k,j}_{V,U} ,\\
& = \sum_{i} \pare{ C^{i,i}_{U,U} C^{i,i}_{V,V} - \pare{ C^{i,i}_{U,V} }^2} \underbrace{-\sum_{j \neq k}  \pare{C^{j,k}_{U,V}}^2}_{\leq 0}.
\end{align*}
Then, by definition of the maximal eigenvalue of $C_{U,U}$:
\begin{align*}
  \sum_{i} C^{i,i}_{U,U} C^{i,i}_{V,V} - \pare{ C^{i,i}_{U,V} }^2 & \leq \lambda_{\rm max} \p{C_{U,U}} \pare{\Tr\pare{C_{V,V}} - \sum_{i} \fracd{\pare{ C^{i,i}_{U,V} }^2 }{C^{i,i}_{U,U} } }, \\
\end{align*}
so that by Cauchy-Schwarz inequality
\begin{align*}
 \pare{\sum_{i}  C^{i,i}_{U,V} }^2 \leq \pare{ \sum_{i} \fracd{\pare{ C^{i,i}_{U,V} }^2 }{C_{U,U}^{i,i} } } \times  \pare{ \sum_{i} C_{U,U}^{i,i}}, 
\end{align*}
and we eventually get
\begin{align*}
  \Tr\pare{C_{U,U} C_{V,V}-C_{U,V} C_{V,U} }    \leq \fracd{\lambda_{\rm max} \p{C_{U,U}}}{\Tr\pare{C_{U,U}}} \pare{\Tr\pare{C_{V,V}} \Tr\pare{C_{U,U}} - \Tr \pare{C_{U,V} }^2}.
\end{align*}
The general case of degenerate eigenvalues is obtained by density.
\end{proof}
The expansion of the average square paralellogram area is detailed in the next lemma.
\begin{Lem}\label{lem:eq_al}
Let $U$ and $V$ be two centered random vector in $\R^d$. Let $(U_\ast,V_\ast)$ be a i.i.d. copy. 
Then we have the following decomposition:

\begin{align}\label{eq:eq_al}
 & \E\pare{ \abs{U-U_\ast}^2\abs{V-V_\ast}^2 - \pare{ \pare{U-U_\ast} \cdot \pare{V-V_\ast}}^2  } \nonumber \\
& \quad  = \underbrace{\E\pare{ \abs{U}^2\abs{V}^2 - \pare{U \cdot V}^2 }}_{\geq 0} +  \underbrace{\Tr\pare{ \pare{C_{U,V} -C_{V,U}}\pare{C_{V,U} -C_{U,V}}}}_{\geq 0} \nonumber\\
& \qquad + 2\underbrace{ \pare{  \Tr\pare{C_{U,U} } \Tr\pare{ C_{V,V} } - \Tr \pare{C_{U,V} }^2 - \Tr\pare{C_{U,U} C_{V,V} } + \Tr \pare{C_{U,V} C_{V,U} } }}_{ \dps \mathop{\geq}^{(\text{Lemma~\ref{lem:trace}})} \pare{1-\min\pare{ \frac{\lambda_{\rm max} \p{C_{U,U}}}{\Tr \pare{ C_{U,U} } } ,\frac{\lambda_{\rm max} \p{C_{V,V}}}{\Tr \pare{ C_{V,V} }}} } \pare{\Tr\pare{C_{U,U} } \Tr\pare{ C_{V,V} } - \Tr \pare{C_{U,V} }^2}} .
\end{align}
\end{Lem}
\begin{proof} Before computing terms, recall that if $M,N$ are two square matrices, then
\[
\Tr\pare{M N} = \Tr\pare{NM} =\Tr\pare{M^T N^T}= \Tr\pare{N^T M^T}.
\]
Let us expand the alignement functional (the left hand side of~\eqref{eq:eq_al}). We have first,
  \begin{align*}
    \E\pare{ \abs{U-U_\ast}^2\abs{V-V_\ast}^2 } &= 2 \E\pare{\abs{U}^2\abs{V}^2} +2\E\pare{\abs{U}^2} \E \pare{\abs{V}^2}   + 4  \E\pare{U \cdot U_\ast \, V \cdot V_\ast} + 8 \times 0\\
&=  2 \E\pare{\abs{U}^2\abs{V}^2} + 2 \Tr\pare{C_{U,U}} \Tr\pare{C_{V,V}}  + 4 \Tr \pare{ C_{U,V} C_{V,U}},
  \end{align*}
and second,
\begin{align*}
   &\hspace{-1cm} \E\pare{\pare{ \pare{U-U_\ast} \cdot \pare{V-V_\ast}}^2} \\
=& 2 \E\pare{U \cdot V^2} +  2 \E\pare{U \cdot V}^2 +  2 \E \pare{U \cdot V_\ast^2} +  2 \E \pare{U_\ast \cdot  V \, U \cdot V_\ast} + 8 \times 0 \\
=& 2 \E\pare{U \cdot V^2} + 2 \Tr\pare{C_{U,V}}^{2}  + 2 \Tr \pare{C_{U,U}C_{V,V}} + 2 \Tr \pare{C_{U,V}^2}. 
  \end{align*}
On the other hand, 
\begin{align*}
    \Tr\pare{ \pare{C_{U,V} -C_{V,U}}\pare{C_{V,U} -C_{U,V}}} = -2 \Tr \pare{C^2_{U,V}} + 2\Tr \pare{C_{U,V}C_{V,U}},
  \end{align*}
and the result then follows.
\end{proof}
Two remarks.
\begin{Rem}
For \emph{normalized} ($\E(\abs{U}^2)=\E(\abs{V}^2)=1$) random vectors,
$
  \Tr \pare{ C_{U,U} } = \Tr \pare{C_{V,V} } = 1,
$
and
\begin{equation*}\label{eq:norm_coupl}
  \Tr\pare{C_{U,U} } \Tr\pare{ C_{V,V} } - \Tr \pare{C_{U,V} }^2 = 1- \E\pare{U \cdot V}^2 .
\end{equation*}
Then Lemma~\ref{lem:trace} and~\ref{lem:eq_al} immediately yield Theorem~\ref{the:fund_ineq}.
\end{Rem}
\begin{Rem}
  Assuming that  $\Tr \pare{C_{V,V} } = 1$ and $C_{U,U} = \fracd{1}{d} \Id$ (isotropy), then~\eqref{eq:eq_al} becomes
\begin{align}\label{eq:spe_2}
 &  1-\E\pare{U \cdot V}^2  \leq \frac{d}{d-1} \E\pare{ \abs{U-U_\ast}^2\abs{V-V_\ast}^2 - \pare{ \pare{U-U_\ast} \cdot \pare{V-V_\ast}}^2  }
\end{align}
Moreover, a sufficient condition for equality in~\eqref{eq:spe_2} is given by strongly istropic distributions with co-linear coupling, defined by the fact that the lengths $(\abs{U},\abs{V}) \in \R_+^2$ are independant of the identically coupled and uniformly distributed direction $ \fracd{U}{\abs{U}} (= \fracd{V}{\abs{V}} \, \as)$.
\end{Rem}

\subsection{Two counter-examples}\label{sec:cex}
Let us finally present two negative results, that demonstrates, in coupling / coupling creation inequalities, the { \bf necessity of sub-exponential estimates } on the one hand, and the  { \bf necessity to resort on higher moments of velocity distributions } on the other hand. We give the counter-examples in the form lemmas, with proofs. In both cases, we consider a coupled distribution in the form of random variables
\begin{equation}
  \label{eq:as_cex}
  (U,V)\in \R^{d}\times \R^{d}, \quad U \sim \calN(0,\fracd{1}{d} \Id), \quad \E V =0, \quad \E \abs{V}^2 =1 .
\end{equation}
$(U_{\ast},V_{\ast})$ is an i.i.d. copy.

\begin{Lem}[The necessity of higher order moments]
Let~\eqref{eq:as_cex} holds. Denote the moment of order $1<q<2$:
\begin{equation*}
 m_q \eqdef \E   \pare{   \abs{ V}^{q} }^{1/q},
  \end{equation*}
Then we have 
\[
\begin{cases}
 \lim_{m_q \to 0}   \E   \pare{   \abs{ U - V}^{2} } = 2 \\
  \lim_{m_q \to 0}  {\E\pare{ \abs{U-U_{\ast}}\abs{V-V_{\ast}} - \pare{U-U_{\ast}} \cdot \pare{V-V_{\ast}}}}= 0.
\end{cases}
\]
\end{Lem}
\begin{proof}
  Hölder inequality implies that $\lim_{m_q \to 0} \E   \pare{   \abs{ V-U }^2 }= 2 \neq 0$. Moreover, we have, with again Hölder inequality,
    \begin{align*}
&      \E \pare{ \abs{U -U_{\ast}}\abs{V-V_{\ast}} - \pare{U-U_{\ast}} \cdot \pare{V-V_{\ast}} } \leq 2  \E \pare{ \abs{U-U_{\ast}}\abs{V-V_{\ast}} }  \\
& \hspace{1cm} \leq   \E^{1/p}\pare{ \abs{U-U_{\ast}}^{p}} \E^{1/q}\pare{ \abs{V-V_{\ast}}^q } \xrightarrow[m_q \to +\infty]{} 0.
    \end{align*}
\end{proof}

\begin{Lem}[The necessity of sub-exponential rates]
 Let~\eqref{eq:as_cex} holds. Consider the co-linear coupling
\[
\fracd{V}{\abs{V}} \eqdef \fracd{U}{\abs{U} } \quad \as, 
\]
with moreover the following radial coupling perturbation on some interval $0 < r_- < r_+ < +\infty $:
\[
\abs{V} \eqdef \abs{U} \one{\abs{U} \notin [r_-,r_+]} + \E^{1/2}\pare{\abs{U}^2 \cond{} \abs{U} \in [r_-,r_+]}  \one{\abs{U} \in [r_-,r_+]}.
\]
Then (i) the moments of $\Law(V)$ are uniformly bounded in $r_-,r_+$; (ii) we have the following degeneracy of the coupling - coupling creation estimate
\[
\lim_{r_- \to + \infty}\lim_{r_+ \to r_-} \fracd{\E \pare{ \abs{U-U_{\ast}}\abs{V-V_{\ast}} - \pare{U-U_{\ast}} \cdot \pare{V-V_{\ast}}}} {\E   \pare{   \abs{ V-U }^2}} = 0 .
\]
\end{Lem}
\begin{proof}
  First, for such isotropic ($U$ is normally distributed) and co-linear couplings, the key inequality~\eqref{eq:fund_ineq} is in fact an equality. Denoting:
\[
A \eqdef \fracd{ \pare{U-U_\ast} \cdot \pare{V-V_\ast} }{\abs{U-U_\ast}\abs{V-V_\ast}},
\]
we obtain 
\[
\calR(r_-,r_+) \eqdef \fracd{\E \pare{ \abs{U-U_{\ast}}\abs{V-V_{\ast}} - \pare{U-U_{\ast}} \cdot \pare{V-V_{\ast}}}}{E   \pare{   \abs{ V-U }^2}}  \leq  \fracd{2d-2}{d} \fracd {\E\pare{\abs{U-U_\ast} \abs{V-V_\ast} \, (1-A)}}{\E\pare{\abs{U-U_\ast}^2 \abs{V-V_\ast}^2 \, (1-A^2) }}.
\]
Since $\abs{A} \leq 1$ and $A=1$ when both $\abs{U} \notin [r_-,r_+]$ and $\abs{U_\ast} \notin [r_-,r_+] $, we have
\begin{align*}
  (1-A) \leq  (1-A^2) \pare{ \one{\abs{U} \in [r_-,r_+]  } + \one{\abs{U_\ast} \in [r_-,r_+]  }} \, \as, \\
2 (1-A^2) \geq  (1-A^2) \pare{ \one{\abs{U} \in [r_-,r_+]  } + \one{\abs{U_\ast} \in [r_-,r_+]  }} \, \as,
\end{align*}
and the smoothness of Gaussian density yields
\[
\lim_{r_+ \to r_-} \calR(r_-,r_+) \leq  \fracd{4(d-1)}{d} \fracd{\E\pare{\abs{R_U-U_\ast} \abs{R_U-U_\ast}}}{\E\pare{\abs{R_U-U_\ast}^2 \abs{R_U-U_\ast}^2} },
\]
where $R_U$ is distributed uniformly on the sphere with radius $r_-$ and is independant of $U_\ast$. In the limit $r_- \to + \infty$, dominated convergence implies $\lim_{r_+ \to r_-} \calR(r_-,r_+)  = O(r_{-}^{-2} )$, hence the result.
\end{proof}

\section{Proofs}\label{sec:proofs}

\subsection{Applying Hölder's inequality}\label{sub:holder}
In this section, we prepare the trend to equilibrium analysis. For this puprose, we will apply Hölder's inequality two times to the special inequality~\eqref{eq:fund_ineq}; a first time with respect to particle averaging $\bracket{\, . \, }_N$, and second time with respect to the expectation $ \E(\, . \,) $.

\begin{Lem}\label{lem:weak_ineq}
  Let $(\pN{u},\pN{v}) \in (\R^d\times \R^d)^N$ satisfying the centering and normalization condition~\eqref{eq:cons} ($\bracket{u}_{N} = \bracket{v}_{N} =0$ and $ \bracket{\abs{u}^2}_{N} = \bracket{\abs{v}^2}_{N} =1$); as well as the positive correlation assumption $\bracket{ u \cdot v }_{N} \geq 0$. For any $\delta >0 $, and $p,q >1$ with $1/p + 1/q =1$, we have the inequality:
\begin{align*}
c_{\delta,p}(u,v)  \leq \frac{1}{2} \frac{ \cc\p{u,v} } { \bracket{ \abs{\pN{u}-\pN{v}}^2  }_N^{1 + 1/2\delta} } ,
\end{align*}
with
 \begin{align*}
& c_{\delta,p}(u,v)  =  \\ 
& \quad k_{1,\delta } \min\pare{\kappa_{ \bracket{\pN{u} \otimes \pN{u}}_N}  ,\kappa_{ \bracket{\pN{v} \otimes \pN{v}}_N}  }^{-1-1/2\delta} \bracket{ \abs{\pN{u}-\pN{u}_\ast}^{2p(1+\delta)} }_N ^{-1/2p\delta} \bracket{ \abs{\pN{v}-\pN{v}_\ast}^{2q(1+\delta)} }_N ^{-1/2 q \delta},
  \end{align*}
and constant
\[
k_{1,\delta} = 2^{-3 -1/2\delta} \frac{d-2}{d-1} \frac{\p{ 1 + \delta}^{1 + 1/\delta}}{\p{ 1 + 2 \delta}^{1+1/2\delta}} .
\]
Note that $c_{\delta,p}(u,v)  = 0 $ if and only if $\bracket{\pN{u} \otimes \pN{u}}_N $ and $\bracket{\pN{v} \otimes \pN{v}}_N $ are of (minimal) rank $1$.
\end{Lem}

\begin{proof}
  We first apply~\eqref{eq:fund_ineq} on the pair $(\pN{u},\pN{v})$ with respect to the probability space generated by the particle averaging operator $\bracket{ \quad }_N$. We obtain using the positive correlation condition~$\bracket{u \cdot v}_{N} \geq 0 $:

\begin{align*}
& \min \p{\kappa_{\bracket{\pN{u} \otimes \pN{u}_\ast}_N} ,\kappa_{\bracket{\pN{v} \otimes \pN{v}_\ast}_N} }^{-1} \frac12 \bracket{ \abs{\pN{u}-\pN{v}}^2}_N \leq 
 \bracket{ \abs{\pN{u}-\pN{u}_\ast}^2\abs{\pN{v}-\pN{v}_\ast}^2 - \pare{ \pare{\pN{u}-\pN{u}_\ast} \cdot \pare{\pN{v}-\pN{v}_\ast}}^2  }_N \eqdef I.
\end{align*}
Next, let us denote for $\eps \in \set{+1,-1}$
\[
\al_{\eps} = \abs{\pN{u}-\pN{u}_\ast} \abs{\pN{v}-\pN{v}_\ast} + \eps  \pare{\pN{u}-\pN{u}_\ast} \cdot \pare{\pN{v}-\pN{v}_\ast},
\]
and introduce $b = 1 + 2 \delta $, $a =1+1/2 \delta$ so that $1/a +1/b = 1$. Using Hölder inequality yields
\begin{align*}
&  I = \bracket{  \al_{+} \al_{-}  }_N  = \bracket{  \al_{+}\al_{-}^{1/b}  \al_{-}^{1/a}  }_N  \leq \bracket{ \al_{+}^{b} \al_{-}}_N^{1/b}  \bracket{  \al_{-}}_N ^{1/a} .
\end{align*}
The elementary (sharp) inequality 
\[
(1+\theta)(1-\theta)^{1/b}\leq b\pare{\fracd{2}{b+1}}^{1/b+1} \qquad \forall \theta \in [-1,1],
\]
used for $\theta = \pare{\pN{u}-\pN{u}_\ast} \cdot \pare{\pN{v}-\pN{v}_\ast} / \abs{\pN{u}-\pN{u}_\ast}\abs{\pN{v}-\pN{v}_\ast} $ then yields
\begin{align*}
I  \leq  b\pare{\fracd{2}{b+1}}^{1/b+1} \bracket{\abs{\pN{u}-\pN{u}_\ast}^{1+b}\abs{\pN{v}-\pN{v}_\ast}^{1+b} }_N^{1/b} \bracket{\al_{-}}^{1/a}_N.
\end{align*}
Finally, remarking that $b \pare{\fracd{2}{b+1}}^{1/b + 1} = \fracd{1+2 \delta}{(1+\delta)^{\frac{2+2\delta}{1+2\delta}}}$, and applying again Hölder inequality with $1/p+1/q =1$ yields
\[
I^{b=1+2\delta}  \leq \frac{(1+2 \delta)^{1+2\delta}}{(1+\delta)^{{2+2\delta}}} \bracket{\abs{\pN{u}-\pN{u}_\ast}^{2p(1+\delta)} }_N^{1/p} \bracket{ \abs{\pN{v}-\pN{v}_\ast }^{2q(1+\delta)} }_N^{1/p} \bracket{\al_{-}}^{b/a = 2\delta }_N.
\]
The result follows.
\end{proof}

The latter lemma implies the following corollary on random vectors with positive correlations.

\begin{Lem}\label{lem:hold_0}
 Let $(U,V) \in (\R^d)^N \times (\R^d)^N $ satisfying the centering and normalization condition (conservation laws)~\eqref{eq:cons}, as well as the positive correlation assumption $\bracket{ U \cdot V }_{N} \geq 0 \as$. Denote $\Law(U)  = \pi_{1} $, $\Law(V)  = \pi_2$. Let $\delta > 0 $ and $ q > 1$ be given. We have:
\begin{equation}
  \label{eq:cc_plaw}
  c_{\delta,p}(\pi_1,\pi_2) \leq  \frac{1}{2} \frac{ \E \, \cc\p{U,V} } { \E \p{ \bracket{ \abs{\pN{U}-\pN{V}}^2  }_N }^{1 + 1/2\delta} } ,
\end{equation}
where in the above:
\begin{align*}
 c_{\delta,p}(\pi_1,\pi_2)  =  k_{2,\delta} \E \p{ \kappa_{ \bracket{\pN{U} \otimes \pN{U}}_N} ^{p(1+2\delta)} \bracket{ \abs{\frac{U-U_\ast}{\sqrt{2}}}^{2p(1+\delta)} }_N }^{-1/2p\delta}  \E \p{  \bracket{ \abs{\pN{V}}^{2q(1+\delta)} }_N }^{-1/2 q \delta},
  \end{align*}
with constant
\[
k_{2,\delta} = {2^{-9/2 - 2/\delta}}\frac{d-2}{d-1} \frac{\p{ 1 + \delta}^{1 + 1/\delta}}{\p{ 1 + 2 \delta}^{1+1/2\delta}} .
\]
\end{Lem}
\begin{proof}
  Hölder's inequality with $1/p + 1/q =1$ implies
\[
\E\p{d(U,V)^{2}} \leq \E\p{d(U,V)^{2q} \cc(U,V)^{-q/p}}^{1/q} \E\p{\cc(U,V)}^{1/p}.
\]
Then the result follows by taking $p=1+1/2\delta$ and $q=1 + 2 \delta$, together with the standard inequality for $n \geq 1$:
\[
\abs{v-v_\ast}^n \leq 2^{n-1} \p{ \abs{v}^n + \abs{v_\ast}^n  }.
\]
\end{proof}

\subsection{Moments of a Wishart's eigenvalues ratio}

In this section, we detail an estimate of the moments of the random condition number $\kappa(\bracket{U \otimes U}_N) = (1- \lambda_{\rm max}(\bracket{U \otimes U}_N))^{-1} $, when $U$ is distributed uniformly on the sphere $\S^{dN-d-1}$ defined by the centering and normalization condition used in the present paper: $\bracket{U \otimes U}_N=0$ , and $\Tr \p{\bracket{U \otimes U}_N} =1 $.

We start with a well-known result from random matrix theory:
\begin{Lem}\label{lem:wish}
  Let $U \sim {\rm unif}_{\S^{dN-d-1}}$, the uniform probability distribution on the sphere $\S^{dN-d-1}$ defined by $\bracket{U \otimes U}_N=0$ , and $\Tr \p{\bracket{U \otimes U}_N} =1 $. The order statistics $ 0 \leq L^N_1 \leq \ldots \leq L^N_d \leq 1$ of the eigenvalues of $ \bracket{U \otimes U}_N $ are distributed according to
  \begin{equation}
    \label{eq:wish_ens}
    \frac{1}{Z(N-1,d)} \prod_{i=1}^d l_i^{(N-2-d)/2} \prod_{1\leq i < j \leq d} \p{l_j - l_i } \one{ 0 \leq l_1 \leq \ldots \leq l_d \leq 1} {\rm vol}_{l_1+ \ldots + l_d =1}(\d l),
  \end{equation}
where ${\rm vol}$ is induced by the canonical euclidean structure of $\R^d$, and the normalization constant satisfies:
\[
Z(N-1,d) \eqdef \pi^{d/2} \Gamma\p{(N-1)d/2}  / \prod_{i=0}^{d-1} \Gamma\p{ (d-i)/2} \Gamma\p{ (N-1-i)/2} .
\]
\end{Lem}
\begin{proof}
Step~(i). The point is to rewrite the distribution of $\bracket{U \otimes U}_N $ as a rescaled Wishart distribution, using a sample co-variance matrix associated to normal idependent random variables. For this purpose, denote:
\[
(G_{(1)}, \cdots , G_{(N)}) \sim \calN(d,N),
\]
a normal random matrix of size $d \times N$ with centered and normalized i.i.d. entries. Then Cochran's theorem ensures that
\[
(G_{(1)} - \bracket{G}_N, \cdots , G_{(N)} - \bracket{G}_N),
\]
is a normal vector with identity co-variance in the sub-vector space defined by $\bracket{g}_N = 0$, so that 
\[
 \frac{1}{\bracket{\abs{ G- \bracket{G}_N }^2}_N}  (G_{(1)} - \bracket{G}_N, \cdots , G_{(N)} - \bracket{G}_N) \in \S^{dN-d-1},
\]
is uniformly distributed.

On the other hand,  $\bracket{ \p{ G- \bracket{G}_N} \otimes \p{ G- \bracket{G}_N} } $ (called a sample co-variance of a multivariate normal distribution), and is well-known to be distributed according to a Wishart distribution of dimension $d$ with $N-1$ degrees of freedom, denoted $\calW_d(N-1)$ see for instance~\cite{And58}. 

As a consequence, the spectrum of $\bracket{U \otimes U}_N$, $0 \leq L^N_1 \leq \ldots \leq L^N_d \leq 1$, is distributed according to
\[
 0 \leq \frac{M^N_1}{\sum_i M^N_i} \leq \ldots \leq \frac{M^N_d}{\sum_i M^N_i} \leq 1 ,
\]
where $0 \leq M^N_1 \leq \ldots \leq M^N_d \leq 1$ is the spectrum of a Wishart distribution $\calW_d(N-1)$. 

Step~(ii). The spectrum of Wishart distributions, and related quantities, can be computed explicitly. It is a classical topic of random matrix theory (see e.g.~\cite{And58,AndGuiZei10,SchKriChat73}.
\end{proof}
This leads to the following property
\begin{Lem}\label{lem:wish_int}
  For any $p \geq 1 $, and $L_d^N$ as in Lemma~\ref{lem:wish}. Then,
\[
\frac{d-1}{d} \geq \E \p{ (1 - L^N_d )^{- p}}^{-1 / p} \mathop{\sim}_{N \to + \infty} \frac{d-1}{d}.
\]
\end{Lem}
\begin{proof}
  By Jensen's inequality
\[
(1 - L^N_d )^{- p} =( L^N_2 + \cdots + L^N_d )^{- p}  \leq \frac{1}{(d-1)^{p}}\underbrace{\prod_{i=2}^d (L_i^N)^{-p/(d-1)}}_{= (L_d^N)^{p/(d-1)}\prod_{i=1}^d (L_i^N)^{-p/(d-1)}},  
\]
so that, using the explicit formula~\eqref{eq:wish_ens},
\[
\E \p{ (1 - L^N_d )^{- p}} \leq  \frac{Z(N_{\rm eff},d)}{Z(N,d)} \E \p{(L_d^{N_{\rm eff} })^{ p/(d-1)}},
\]
where $N_{\rm eff} = N -2p / (d-1) $ is an effective sample size. Then by dominated convergence
\[
\E \p{(L_d^{N_{\rm eff} })^{ p/(d-1)}}  \limop{\to}_{N \to + \infty} d^{-p/(d-1)};
\]
and by Stirling's formula:
\[
\frac{ \Gamma\p{(N-1)d/2}}{ \Gamma\p{(N_{\rm eff}-1)d/2}} \mathop{\sim}_{N \to + \infty} \p{(N-1)d/2}^{(N-N_{\rm eff})d/2},
\]
as well as
\[
\frac{ \Gamma\p{(N-1)/2 -i/2}}{ \Gamma\p{(N_{\rm eff}-1)/2 -i/2}} \mathop{\sim}_{N \to + \infty} \p{(N-1)/2}^{(N-N_{\rm eff})/2},
\]
so that
\[
\frac{Z(N_{\rm eff}-1,d)}{Z(N-1,d)} \limop{\to}_{N \to + \infty}  d^{p d/(d-1)}.
\]
The result follows.
\end{proof}

\subsection{Proof of Theorem~\ref{th:main}}\label{sub:proof_main}



{\bf Step 1: Initial condition.} Let $\pi_t$ denotes the distribution of the Kac's particle system at time $t \geq 0$. We fix $t \geq 0$, and following Lemma~\ref{lem:couplrep}, we take as a new initial condition
\[
\p{\opt{U}_{t,(1)}, \opt{V}_{t,(1)}, \ldots , \opt{U}_{t,(N)}, \opt{V}_{t,(N)} } = \p{U_{t,(\Sigma(1))}, {V}_{t,(\Sigma(1))}, \ldots ,{U}_{t,(\Sigma(N))}, {V}_{t,(\Sigma(N))} }  \in (\R^d \times \R^d)^N,
\]
where $(U_t,V_t,\Sigma) \in (\R^d\times \R^d)^N \times {\rm Sym}_N$ is a representative of the $d_{{\rm sym}, W_2}$-optimal coupling between $\pi_\infty = \Law(U_t)$ and $\pi_t = \Law(V_t)$. We denote $\opt{\pi}_t = \Law(\opt{V}_t)$ and remark that although $\pi_t = \tilde{\pi}_t$, the distributions $ \opt{\pi}_t $ and $\pi_t $ have the same permutation invariant moments or observables. Finally, Lemma~\ref{lem:couplrep} implies the positive correlation assumption $\bracket{\opt{U}_{t} \cdot \opt{V}_{t} }_N \geq 0 \as $.

{\bf Step 2: Propagation.} Let $b_0 = \int \beta(\d \theta) < + \infty$ a given angular cut-off. Consider the solution $h \mapsto (\opt{U}_{t+h},\opt{V}_{t+h})$ of the coupled Kac's particle system with the initial condition~$(\opt{U}_{t}, \opt{V}_{t} )$ and angular cut-off $b_0$. We get:
\begin{align*}
  d_{{\rm sym},W_2} \pare{\pi_{t} ,\pi_\infty}^2 = \E \bracket{ \abs{\opt{U}_{t}-\opt{V}_{t}}^2 }_N = \E \bracket{ \abs{\opt{U}_{t+h}-\opt{V}_{t+h}}^2 }_N  + \int_{0}^{h} \E \p{ \cc \p{\opt{U}_{t+h'},\opt{V}_{t+h'}} } \d h' .
\end{align*}
 Next, by definition of the Wasserstein distance:
\[
\E \bracket{ \abs{\opt{U}_{t+h}-\opt{V}_{t+h}}^2 }_N \geq d^2_{W_2} \pare{\opt{\pi}_{t+h} ,\pi_\infty};
\]
Now, since the coupling distance is almost surely decreasing (and energy is conserved), the positive correlation condition propagtes so that 
\[
\bracket{\opt{U}_{t+h'} \cdot \opt{V}_{t+h'} }_N \geq 0 \as \quad \forall h' \geq 0,
\] 
and we can apply Lemma~\ref{lem:hold_0} to get
\begin{align*}
  \label{eq:5}
  \E \p{ \cc \p{\opt{U}_{t+h'},\opt{V}_{t+h'}} } & \geq 2 c_{\delta,q,N}(\opt{\pi}_{t+h'},\pi_\infty) \p{  \E \bracket{ \abs{\opt{U}_{t+h'}-\opt{V}_{t+h'}}^2 }_N }^{1 +1/2\delta} \\
& \geq 2 c_{\delta,q,N}(\opt{\pi}_{t+h'},\pi_\infty)  d_{W_2} \pare{\opt{\pi}_{t+h} ,\pi_\infty} ^{2 +1/\delta}.
\end{align*}

Then, it is possible take the limit $b_0 \to +\infty $ of vanishing cut-off in the above inequality. Indeed, it is known (see for instance~\cite{EthKur85}) that any well-defined Makov process (including diffusion and Levy processes) on a manifold can be approximated by a bounded jump process, in the sense of convergence of distribution on trajectory (Skorokhod) space. As a consequence, we can assume that if $\opt{\pi}_t$ is a solution without angular cut-off, it is possible to construct a sequence $\opt{\pi}^{b_{0,n}}_t$ with angular cut-off such that $\opt{\pi}^{b_{0,n}}_t \to \opt{\pi}_t$ in distribution. Since the Wasserstein distance $d_{W_2}$ metrizes weak convergence, the state space $\S^{Nd-N-1}$ is compact, and moments are continuous observables, it is possible to remove the angular cut-off in the above inequality.

Finally, using the inequality
\[
d_{{\rm sym},W_2} \pare{\pi_{t+h} ,\pi_\infty} = d_{{\rm sym},W_2} \pare{\opt{\pi}_{t+h} ,\pi_\infty} \leq d_{W_2} \pare{\opt{\pi}_{t+h} ,\pi_\infty} ,
\]
we obtain the result:
\[
\frac{\d }{\d t} d^2_{{\rm sym},W_2} \pare{\pi_{t} ,\pi_\infty} \leq - c_{\delta,q,N}(\pi_t,\pi_\infty) d^2_{{\rm sym},W_2} \pare{\pi_{t} ,\pi_\infty} .
\]

{\bf Step 3: $N$-uniform control of the constant.}

Let us introduce the notation
\[
c_{\delta,q,N}(\pi_t,\pi_\infty) = k_{\delta , q , N} \E \p{  \bracket{ \abs{\pN{V_t}}^{2q(1+\delta)} }_N }^{-1/2 q \delta} ,
\]
with 
\begin{align}
k_{\delta , q , N} = {2^{-9/2-2\delta}} \frac{d-2}{d-1} \frac{\p{ 1 + \delta}^{1 + 1/\delta}}{\p{ 1 + 2 \delta}^{1+1/2\delta}}  \times \E \p{ \kappa_{ \bracket{\pN{U} \otimes \pN{U}}_N} ^{p(1+2\delta)} \bracket{ \abs{\frac{U-U_\ast}{\sqrt{2}}}^{2p(1+\delta)} }_N }^{-1/2p\delta} , \label{eq:k_main}
\end{align}
where $U$ is distributed uniformly on sphere $\S^{Nd-N-1}$. Using the integrability property in Lemma~\ref{lem:wish_int}, and the well-known convergence with large dimension of the uniform distribution on spheres towards Gaussian distributions, we obtain:
\[
\lim_{N \to + \infty} \E \p{ \kappa_{ \bracket{\pN{U} \otimes \pN{U}}_N} ^{p(1+2\delta)} \bracket{ \abs{\frac{U-U_\ast}{\sqrt{2}}}^{2p(1+\delta)} }_N }^{-1/2p\delta} = \p{ \frac{d-1}{d} }^{1+1/2\delta}
\E (\abs{G_d}^{2p(1+\delta)})^{-1/2p \delta}.
\]
where $G_d$ a $d$-dimensional centered normal distribution with $\E (\abs{G_d}^2)=1$. The result follows.

\subsection{Order~$4$ moment control}
In this section, the evolution of the radial order~$4$ moment of the Kac's particle system is calculated.
\[
(v,v_\ast,v',v'_\ast) \in (\R^d)^4
\]
denotes a solution of the collision mapping~\eqref{eq:collision}, and we will use again the following quantities:
\begin{align*}
  \begin{cases}
\dps n_v \eqdef (v-v_\ast)/\abs{v-v_\ast} , \\[6pt]
 \dps    2 s_v \eqdef v + v_\ast  , \\[6pt]
\dps 2 d_v \eqdef v - v_\ast . \\[6pt]
  \end{cases}
\end{align*}
 By Lemma~\eqref{lem:azim}, it is possible to: (i) pick $(n_v,m_v) \in (S^{d-1})^2$, an orthonormal pair such that $s_v \in \Span\p{n_v,m_v} $; (ii) consider spherical coordinates such that:
\[
n'_v = \cos \theta \, n_v + \sin \theta \cos \ph \, m_v + \sin \theta \sin \ph \,  l_v \in \S^{d-1},
\]
where $l_v \in \S^{d-1} \cap \Span\p{n_v,m_v}^{\perp} $; (iii) and write the collision kernel in $(\theta,\ph,l_v)$-coordinates as:
\[
b(n_v, \d n_v') \equiv {\rm unif}_{\S^{d-1} \cap \Span\p{n_v,m_v}^{\perp}  }(\d l_v) \sin^{d-3}(\ph) \frac{\d \ph}{ w_{d-3}} \beta( \d \theta).
\]

\begin{Lem} Under the normalized Levy condition~\eqref{eq:Levy}, the the post-collisional order~$4$ radial  moment satisfies:
  \begin{align}\label{eq:mom4}
\Delta_4(v,v_\ast )&\eqdef \frac{1}{2}\int_{\S^{d-1}} \abs{v'}^4+ \abs{v'_\ast}^4 - \abs{v}^4 - \abs{v_\ast}^4 b(n_v, \d n_v') \nonumber \\
& = - \frac14 \p{ \abs{v}^4 + \abs{v_\ast}^4 } + \frac{d+1}{2(d-1)}\abs{v}^2 \abs{v_\ast}^2 - \frac{1}{d-1} \p{v \cdot v_\ast }^2.
  \end{align}
\end{Lem}
\begin{proof}
 Straightforward calculation (that can be double-checked using the stationarity under the Boltzmann kernel $b$ of product Gaussian distributions of the form $\calN( \d v) \otimes \calN(\d v_\ast) $). See also the classical paper~\cite{IkeTru56} for a general treatment of moments for Maxwell molecules.
\end{proof}
It then follows:
\begin{Lem}
  Let $(V_t)_{t \geq 0} \in \p{\R^{d}}^N$ a centered and normalized conservative Kac's particle system. Then for any $t \geq 0$:
  \begin{equation*}
    \E\bracket{\abs{V_t}^4}_N \leq \e^{-  t / 2 } \p{ \E\bracket{\abs{V_0}^4}_N -\frac{d+2}{d} } + \frac{d+2}{d}.  
  \end{equation*}
Moreover, denoting (convention: $ \ln a = -\infty$ if $a < 0$) :
\[
t_\ast \eqdef  2 \p{\ln \p{ \frac{d}{d+2} \E\bracket{\abs{V_0}^4}_N - 1}  }^+
\]
then for any $\gamma , t >0$:
  \begin{equation}
    \label{eq:cons4}
    \int_0^{t} \p{ \E \bracket{\abs{V_s}^4}_N }^{-\gamma} \d s \geq \p{\frac{2d+4}{d} }^{-\gamma}  \p{ t- t_\ast}^+ .
  \end{equation}
\end{Lem}
  \begin{proof}
    The order~$4$ formula~\eqref{eq:mom4} implies that
    \begin{align*}
 \frac{\d}{\d t} \E \bracket{\abs{V_t}^4}_N &= \bracket{\Delta_4(V_t,V_{t,\ast})}_N \\
&= - \frac12 \E \bracket{\abs{V_t}^4}_N + \frac{d+1}{2(d-1)}  - \frac{1}{d-1} \E  \underbrace{ \Tr\p{ \bracket{V_t \otimes V_t }^2_N } }_{ \dps \mathop{\geq}^{\text{Cauchy-Schwarz}} \Tr (\Id)= d  } \\
& \leq - \frac12 \E \bracket{\abs{V_t}^4}_N  + \frac{d+2}{2d},
    \end{align*}
so that the first result follows by Gronwall's argument.

Now if $t_\ast =0$, then
$
\E\bracket{\abs{V_0}^4}_N  \leq \frac{2d+4}{d} .
$
Otherwise, since $\e^{- t_\ast / 2 } \p{ \E\bracket{\abs{V_0}^4}_N -\frac{d+2}{d} } = \frac{d+2}{d}$ we obtain
\[
\E \bracket{\abs{V_t}^4}_N \leq \frac{2d+4}{d}.
\] 
The result follows.
\end{proof}

\subsection{Proof of Proposition~\ref{pro:order4}}
We can then apply Gronwall's lemma to Theorem~\ref{th:main} with the estimate~\eqref{eq:cons4}. Since $2\delta(1+p) =4 $, we take $\gamma=1/2p\delta =  1/4+1/4\delta $, and find:
\begin{align*}
  c_{\delta,N} = k_{\delta,\frac{2}{1-\delta},N} \p{\frac{2d+4}{d} }^{-1/2 -1/2\delta },
\end{align*}
where in the above $k_{\delta,\frac{2}{1-\delta},N} $ is defined by~\eqref{eq:k_main}. The result follows.






\section*{Acknowledegement}
I thank C. Mouhot, N. Fournier, F. Bolley, and A. Guillin for their helpful comments.
\bibliographystyle{plain}
\bibliography{boltzmann}
\end{document}